\documentclass[a4paper,11pt]{amsart}
\usepackage{amsfonts}
\usepackage{amsthm}
\usepackage{amssymb}
\usepackage{amsmath}
\usepackage{enumerate}
\usepackage{comment}
\usepackage{a4wide}
\usepackage{dsfont}
\usepackage{bbm}
\usepackage{xparse,mathrsfs,mathtools}

\newtheorem{thm}{Theorem}
\newtheorem{lemma}[thm]{Lemma}
\newtheorem{cor}[thm]{Corollary}

\allowdisplaybreaks

\begin{document}

\title[On the distribution of partial quotients of reduced fractions]{On the distribution of partial quotients of reduced fractions with fixed denominator}
\author{Christoph Aistleitner}
\author{Bence Borda}
\author{Manuel Hauke}
\address{Graz University of Technology, Institute of Analysis and Number Theory, Steyrergasse 30/II, 8010 Graz, Austria}
\email{aistleitner@math.tugraz.at}
\email{borda@math.tugraz.at}
\email{hauke@math.tugraz.at}


\renewcommand{\labelenumi}{\alph{enumi})}

\begin{abstract} 
In this paper, we study distributional properties of the sequence of partial quotients in the continued fraction expansion of fractions $a/N$, where $N$ is fixed and $a$ runs through the set of mod $N$ residue classes which are coprime with $N$. Our methods cover statistics such as the sum of partial quotients, the maximal partial quotient, the empirical distribution of partial quotients, Dedekind sums, and much more. We prove a sharp concentration inequality for the sum of partial quotients, and sharp tail estimates for the maximal partial quotient and for Dedekind sums, all matching the tail behavior in the limit laws which are known under an extra averaging over the set of possible denominators $N$. We show that the distribution of partial quotients of reduced fractions with fixed denominator gives a very good fit to the Gau{\ss}--Kuzmin distribution. As corollaries we establish the existence of reduced fractions with a small sum of partial quotients resp.\ a small maximal partial quotient.
\end{abstract}
	
\maketitle
	
\section{Introduction and statement of results}

The theory of continued fractions plays a central role in Diophantine approximation, and is of major importance in many other mathematical areas such as algorithmics, dynamical systems, and geometry of numbers, to name only a few. Given an integer $N \ge 1$, throughout this paper $\mathbb{Z}_N^*$ denotes the set of all integers $1 \le a \le N$ which are coprime with $N$. For any $N \ge 2$ and $a \in \mathbb{Z}_N^*$, we write $a/N = [0;a_1,a_2,\ldots,a_r]$ for the continued fraction expansion of $a/N$, with the convention that we always use the canonical representation for which $a_r > 1$. We are interested in statistics such as the sum of partial quotients, the maximal partial quotient, or the number of partial quotients within a certain test window, when the denominator $N$ remains fixed and $a$ runs through the set $\mathbb{Z}_N^*$.

Let
$$
S(a/N)=\sum_{i=1}^r a_i
$$ 
denote the sum of partial quotients. This sum plays a crucial role in estimates for the algorithmic complexity of many important algorithms of an arithmetic nature, as well as for the discrepancy of Kronecker-type sets $(\frac{n a \textup{ mod } N}{N})_{1 \leq n \leq N}$ (in dimension 1) and lattice point sets $(\frac{n}{N},\frac{n a \textup{ mod } N}{N})_{1 \leq n \leq N}$ (in dimension 2). A number $a \in \mathbb{Z}_N^*$ for which $S(a/N)$ is small generates a set of ``good lattice points''. The theory of quadrature rules for numerical integration based on such good lattice point sets was initiated by Hlawka \cite{hla} and Korobov \cite{koro} in the 1960s, and further developed by many other authors since then. For more details on this topic we refer the reader to \cite{dkp,nied}.

These examples motivate questions such as: Given a denominator $N$, what is the smallest possible order of $S(a/N)$ among all reduced fractions $a/N$? What is the typical order of $S(a/N)$, and how likely is the sum of partial quotients  to deviate significantly from this typical behavior? How many partial quotients of a given size does a reduced fraction $a/N$ typically have? What is the smallest possible order of the maximal partial quotient of $a/N$? Many questions of this type are addressed by the theorems in this paper.

The average value of $S(a/N)$ was found independently by Panov \cite{panov} and Liehl \cite{liehl}, who showed that
\begin{equation} \label{typ_ord} \frac{1}{\varphi (N)} \sum_{a \in \mathbb{Z}_N^*} S \left( \frac{a}{N} \right) \sim \frac{6}{\pi^2} (\log N)^2 \qquad \textrm{as } N \to \infty ,
\end{equation}
where $\varphi$ is the Euler totient function. A similar result for the average over all $1 \le a \le N$ (without the coprimality condition) was obtained by Yao and Knuth \cite{knyao}, who interpreted this as the average number of steps required by the subtractive greatest common divisor algorithm (see also the discussion in Section 4.5 of Knuth's book \cite{knuth}).

Our first theorem is a concentration inequality for the sum of partial quotients. The main message is that the typical order of $S(a/N)$ is $(12/\pi^2) \log N \log \log N$, and that there is only a small proportion of fractions for which $S(a/N)$ deviates significantly from this typical order. The fact that the average order $(\log N)^2$ in equation \eqref{typ_ord} is much larger than the typical order $\log N \log \log N$ in our theorem is caused by this small proportion of exceptional fractions, which might possess one or more unusually large partial quotients. In probabilistic terminology, this deviation between the average order and the typical order is caused by the fact that the distribution of partial quotients is heavy-tailed.
\begin{thm}\label{Stheorem} Let $C>0$ be arbitrary. For all $N \geq 3$ and $0<t \le (\log N)^C$, we have
\[ \frac{1}{\varphi (N)} \left| \left\{ a \in \mathbb{Z}_N^* \, : \, \left| S \left( \frac{a}{N} \right) - \frac{12}{\pi^2} \log N \log \log N \right| \ge t \log N \right\} \right| \ll \frac{1}{t} \]
with an implied constant depending only on $C$.
\end{thm}

This improves a result of Rukavishnikova \cite{ruka_1,ruka_2}, who proved the upper tail estimate
\begin{equation}\label{ruka}
\frac{1}{\varphi (N)} \left| \left\{ a \in \mathbb{Z}_N^* \, : \, S \left( \frac{a}{N} \right)\geq t  \log N \log \log N \right\} \right| \ll \frac{1}{t}
\end{equation}
in the same range $0<t \leq (\log N)^C$. Rukavishnikova obtained her result by establishing an asymptotic upper bound for the average value of a restricted sum of partial quotients. In contrast, in this paper we establish a precise asymptotics for the average value, together with an upper bound for the variance of such restricted sums, which allows us to deduce the concentration inequality in Theorem \ref{Stheorem}.

Theorem \ref{Stheorem} should be compared to recent results of Bettin and Drappeau \cite{bd2}, who proved that under an additional averaging over the set of all denominators $N$ up to a certain threshold, the distribution of a suitably centered and rescaled version of $S(a/N)$ converges to a limit distribution. More precisely, when considering all reduced fractions $a/N$ in $[0,1]$ for which $N \leq Q$ (the so-called Farey fractions of order $Q$, denoted by $F_Q$), then for all real $t$,
$$
\lim_{Q \to \infty} \frac{1}{|F_Q|} \left| \left\{  \frac{a}{N} \in F_Q \, : \, \frac{S(a/N) - \frac{12}{\pi^2} \log N \log \log N}{\log N} \le t \right\} \right| = F(t) ,
$$
where $F(t)$ is the cumulative distribution function of a specific stable law with stability parameter $1$. In particular, the tail probabilities are of order $1-F(t) \sim c/t$ with some constant $c>0$ as $t \to \infty$, matching the tail estimate in our Theorem \ref{Stheorem} for the case of fixed denominators. Therefore Theorem \ref{Stheorem} is optimal up to the value of the implied constant.

We note that the methods of Bettin and Drappeau are very different from the ones used in the present paper, and rely on a dynamical systems approach which builds upon earlier work of Dolgopyat \cite{dolg} and Baladi and Vall\'{e}e \cite{bave}. It remains unclear if a suitably centered and rescaled version of $S(a/N)$ has a limit distribution when $N$ is fixed, and if such a limiting distributional behavior would have to take into account certain arithmetic properties of $N$ (an arithmetic effect which is possibly lost when taking an additional average over all denominators $N$).

As an immediate consequence of Theorem \ref{Stheorem}, we have the following corollary.
\begin{cor} \label{co1}
For all $N \geq 3$, there exists a reduced fraction $a/N$ for which
$$
S \left( \frac{a}{N} \right) \leq \frac{12}{\pi^2} \log N \log \log N + O(\log N).
$$
\end{cor}
\noindent To put the conclusion of the corollary into perspective, we note that the tail estimate \eqref{ruka} of Rukavishnikova only guarantees the existence of a reduced fraction $a/N$ for which $S(a/N) \ll \log N \log \log N$ with an unspecified implied constant. We also note that a well-known lower bound states that $S(a/N) \gg \log N$ for all reduced fractions $a/N$, so Corollary \ref{co1} is optimal up to the double logarithmic factor.

The next result is a tail estimate for the maximal partial quotient of $a/N$, denoted by
\[ M \left( \frac{a}{N} \right) = \max_{1 \le i \le r} a_i . \]
\begin{thm}\label{Mtheorem} Let $C>0$ be arbitrary. For all $N \geq 3$ and $0<t \le (\log N)^C$, we have
\[ \frac{1}{\varphi (N)} \left| \left\{ a \in \mathbb{Z}_N^* \, : \, M \left( \frac{a}{N} \right) \ge t \log N \right\} \right| \le \frac{12}{\pi^2 t} + O \left( \frac{(\log \log N)^3}{t \log N} \right) \]
with an implied constant depending only on $C$.
\end{thm}
\noindent This should be compared with a result of Hensley \cite[Theorem 1]{hensley}, who found the limit distribution of a rescaled version of $M(a/N)$ when an additional average over $N$ is taken. More precisely, again writing $F_Q$ for the set of Farey fractions of order $Q$, for all $t>0$ we have
$$
\lim_{Q \to \infty} \frac{1}{|F_Q|}  \left| \left\{ \frac{a}{N} \in F_Q \, : \, M\left(\frac{a}{N} \right) \geq t \log N \right\} \right| = 1 - e^{-\frac{12}{\pi^2 t}}.
$$
Note that
$$
1 - e^{-\frac{12}{\pi^2 t}} = \frac{12}{\pi^2 t} + O\left( \frac{1}{t^2} \right) \qquad \text{as $t \to \infty$}, 
$$
so for large $t$ the tail estimate for fixed $N$ provided by our Theorem \ref{Mtheorem} perfectly matches the tail behavior after taking an additional average over all denominators, even including the correct constant factor $12/\pi^2$. As in the discussion following Theorem \ref{Stheorem}, it remains unclear whether the additional average over all denominators is necessary to obtain a limit distribution; that is, whether there exists an analogue of Hensley's theorem which holds for fixed denominators, and whether such a result would be universal, or rather have to account for certain arithmetic properties of $N$. 

From Theorem \ref{Mtheorem} we immediately obtain the following corollary.
\begin{cor}
For all $N \geq 3$, there exists a reduced fraction $a/N$ for which
$$
M(a/N) \leq \frac{12}{\pi^2} \log N + O((\log \log N)^3).
$$
\end{cor}

To the best of our knowledge, the fact that for all $N$ there exists a reduced fraction $a/N$ with $M(a/N) \ll \log N$ was first proved by Korobov \cite{koro} for prime $N$ and by Zaremba \cite{zaremba2} for composite $N$. In \cite{cusick} it was shown that for all $N$ there exists a reduced fraction $a/N$ with $M(a/N) \leq 3 \log N$, which appears to have been the strongest known existence result so far (note for comparison that $12/\pi^2 \approx 1.22$). Zaremba's famous conjecture asserts that actually for all $N$ there is a reduced fraction $a/N$ with $M(a/N) \leq K$, where $K$ is a suitable absolute constant (it is conjectured that $K=5$ is admissible). Zaremba's conjecture was established when $N$ is of very special type; most importantly, Niederreiter \cite{nieder} proved the conjecture in the case when $N$ is a power of a small prime. The conjecture has been confirmed by Bourgain and Kontorovich \cite{bk} for a set of denominators having full asymptotic density; see also \cite{fk,huang}. In the recent preprint \cite{mms_preprint}, the estimate $M(a/N) \ll \log N$ was improved to $M(a/N) \ll \log N / \log \log N$ for a wide range of denominators $N$. However, as a statement concerning \emph{all} possible denominators $N$, Zaremba's conjecture remains wide open. 

Given integers $1 \le b \le c$, let
$$
L_{[b,c]} \left( \frac{a}{N} \right) = \sum_{i=1}^r \mathds{1}_{\{ b \le a_i \le c \}}
$$
denote the number of partial quotients of $a/N$ which fall in the interval $[b,c]$. Since all partial quotients of $a/N$ are at most $N$, choosing e.g.\ $c=N$ corresponds to counting partial quotients in the interval $[b,\infty)$, that is, $L_{[b,N]}(a/N)=\sum_{i=1}^r \mathds{1}_{\{ a_i \ge b \}}$. In our next result we find the precise asymptotics of the average value, and an upper bound for the variance of $L_{[b,c]}(a/N)$. In particular, for large values of $b$, $L_{[b,c]}(a/N)$ concentrates around its mean for the majority of reduced fractions $a/N$.
\begin{thm}\label{Ltheorem} Let $C>0$ be arbitrary. For all $N \geq 3$ and $1 \le b \le c$ with $b \le (\log N)^C$, we have
\begin{equation}\label{Ltexpectedvalue}
\frac{1}{\varphi (N)} \sum_{a \in \mathbb{Z}_N^*} L_{[b,c]} \left( \frac{a}{N} \right) = \mu_{[b,c]} \log N +O \left( \frac{(\log \log N)^3}{b} \right) ,
\end{equation}
and
\begin{equation} \label{Ltvariance}
\frac{1}{\varphi (N)} \sum_{a \in \mathbb{Z}_N^*} \left( L_{[b,c]} \left( \frac{a}{N} \right) - \mu_{[b,c]} \log N \right)^2 \ll \frac{(\log N)^2}{b^4} + \frac{\log N \log \log N}{b}
\end{equation}
with $\mu_{[b,c]}=\frac{12}{\pi^2} \sum_{m=b}^c \log \left( 1+\frac{1}{m(m+2)} \right)$ and implied constants depending only on $C$.
\end{thm}

By a result of Baladi and Vall\'{e}e \cite[Theorem 3(a)]{bave} (see also Bettin and Drappeau \cite[Corollary 3.2]{bd2}), under an extra averaging over the denominators $N$ we once again have a limit distribution: for any fixed $1 \le b \le c$,
\[ \frac{L_{[b,c]}(a/N) - \mu_{[b,c]} \log N}{\sigma_{[b,c]} \sqrt{\log N}} \]
with a suitable $\sigma_{[b,c]}>0$ converges to the standard normal distribution. It remains open whether a similar central limit theorem holds for reduced fractions with a fixed denominator. Avdeeva and Bykovski\u{\i} \cite{a_byk} applied the results of Baladi and Vall\'{e}e to establish (as a special case) concentration of  $L_{[b,c]}$ around its mean, but it is not clear if their method allows the computation of first and second moments as in our theorem.

Note that Theorem \ref{Ltheorem} is precise enough to yield a sensible result even in the case when $b=c=m$, i.e.\ when we are counting the number of occurrences of a particular value $m$ for the partial quotients. In this case, the theorem gives
$$
\frac{1}{\varphi (N)} \sum_{a \in \mathbb{Z}_N^*} \sum_{i=1}^r \mathds{1}_{\{a_i = m\}} = \frac{12}{\pi^2} \log \left( 1+\frac{1}{m(m+2)}  \right) \log N +O \left( \frac{(\log \log N)^3}{m} \right),
$$
and the main term (which is of order $m^{-2} \log N$) dominates the error as long as $m$ is moderate in comparison with $\log N$. It is natural to divide the equation by a factor of $\frac{12 \log 2}{\pi^2} \log N$, which is the asymptotic order of the average length of the continued fraction of a reduced fraction with denominator $N$ (see e.g.\ \cite{porter}), which leads to
$$
\frac{1}{\varphi (N)} \frac{\pi^2}{12 \log 2 \log N} \sum_{a \in \mathbb{Z}_N^*} \sum_{i=1}^r \mathds{1}_{\{a_i = m\}} = \frac{\log \left( 1+\frac{1}{m(m+2)}  \right)}{\log 2} +O \left( \frac{(\log \log N)^3}{m \log N} \right).
$$
The main term on the right-hand side of this equation is the probability of the number $m$ in the Gau{\ss}--Kuzmin distribution, which is the probability distribution governing the distribution of the partial quotients of random reals. Thus, informally speaking, Theorem \ref{Ltheorem} asserts that the distribution of the partial quotients of reduced fractions with fixed denominator is in very good accordance with the distribution coming from the corresponding random model. Versions of the ``expected value'' part of Theorem \ref{Ltheorem}, i.e.\ equation \eqref{Ltexpectedvalue}, have been established by other authors before. In particular, Ustinov \cite{ustinov} obtained a version of \eqref{Ltexpectedvalue} with an error term which is much smaller than ours in terms of $N$, but depends on $b$ and $c$ in an unspecified way. David and Shapira \cite{david_shapira} proved a qualitative version of \eqref{Ltexpectedvalue}, without giving an error term, using methods from ergodic theory. 

Our final result is a tail estimate for Dedekind sums, which are defined for $a \in \mathbb{Z}_N^*$ as
\[ D \left( \frac{a}{N} \right) = \sum_{b=1}^{N-1} \left( \frac{b}{N} - \frac{1}{2} \right) \left( \left\{ \frac{ba}{N} \right\} - \frac{1}{2} \right), \]
where $\{ \cdot \}$ denotes the fractional part function. Dedekind sums are a classical object in analytic number theory, and appear most importantly in the modularity relation of the Dedekind eta function; see for example \cite{rad_gro} for more background. An identity due to Barkan \cite{Barkan} and independently Hickerson \cite{hickerson} relates $D(a/N)$ to the partial quotients via
\begin{equation}\label{hickerson}
\begin{split} D \left( \frac{a}{N} \right) &= \frac{(-1)^r -1}{8} + \frac{1}{12}\left( \frac{a}{N} - (-1)^r[0;a_r,\ldots,a_2,a_1] - S_{\mathrm{alt}} \left( \frac{a}{N} \right) \right)
\\ &= -\frac{1}{12}S_{\mathrm{alt}} \left( \frac{a}{N} \right) + O(1), \end{split}
\end{equation}
where $S_{\mathrm{alt}}(a/N) = \sum_{i=1}^r (-1)^i a_i$ is the alternating sum of partial quotients.

\begin{thm}\label{Dtheorem} Let $C>0$ be arbitrary. For all $N \geq 3$ and $0<t \le (\log N)^C$, we have
\[ \frac{1}{\varphi(N)} \left| \left\{ a \in \mathbb{Z}_N^* \, : \, \left| D \left( \frac{a}{N} \right) \right| \geq t \log N \right\} \right| \ll \frac{1}{t} \]
with an implied constant depending only on $C$.
\end{thm}
Vardi \cite{vardi2} showed that under an additional averaging over the denominators, the normalized Dedekind sum $2 \pi D(a/N)/ \log N$ converges to the Cauchy distribution. Note that our tail estimate $1/t$ in Theorem \ref{Dtheorem} perfectly matches the order of the tail probabilities of the Cauchy distribution. Once again, it is not clear whether the extra averaging over denominators is necessary for the existence of a limit law. 

We add some further references to results which are related to those in the present paper. Rukavishnikova \cite{ruka_1} proved a second moment bound for the average sum of partial quotients when $N$ is fixed and $a \in \mathbb{Z}_N$ (i.e., without a coprimality assumption as in our theorems). In the same setup, Bykovski\u{\i} \cite{byk} proved a second moment bound for the length (which corresponds to choosing $b=1$ and $c=N$ in our Theorem \ref{Ltheorem}) of the continued fraction expansion. The first moment of the length of the continued fraction, in the coprime setup, was calculated with high precision in \cite{by_fro}. Several authors studied problems related to those in the present paper when the average is taken over a (sufficiently large) multiplicative subgroup of $\mathbb{Z}_N^*$, see for example \cite{chang, mu}. 

We now give an outline of the remaining parts of this paper. The key tools in this paper come from sieve theory and uniform distribution (resp.\ discrepancy) theory, of course in combination with results from Diophantine approximation and the theory of continued fractions. In Section \ref{sec_discrep}, we estimate the discrepancy of the set of reduced fractions $\{ a/N \, : \, a \in \mathbb{Z}_N^* \}$ using sieve theory. In Section \ref{sec_weight}, we introduce the main auxiliary object of the paper, the restricted sum $\sum_{i=1}^r \mathds{1}_{\{ \eta \le a_i \le \theta \}} f(a_i)$ with cutoff parameters $\eta$ and $\theta$ and a non-negative, non-decreasing function $f$. We show how the first and second moment of the restricted sum can be approximated by sums of certain weight functions. In Section \ref{sec_expect}, we find the average of such a restricted sum up to a small error, and in Section \ref{sec_variance} we give an upper bound for the variance of the restricted sum. In Section \ref{sec_mainproofs}, we deduce Theorems \ref{Stheorem}, \ref{Mtheorem} and \ref{Ltheorem} from our general results on restricted sums, by specializing to $f(x)=x$ and to $f(x)=1$, respectively. As the overview of the proof strategy indicates, our methods in fact apply to very general sums of the form $\sum_{i=1}^r f(a_i)$, yielding sharp tail estimates or concentration inequalities depending on the growth rate of $f$. Finally, in Section \ref{sec_dedekind} we explain how to modify our approach for alternating sums $\sum_{i=1}^r (-1)^i f(a_i)$, and prove Theorem \ref{Dtheorem} on Dedekind sums.

Before turning to the proofs, we note that our arguments crucially exploit the existence of a certain bijective mapping $\mathbb{Z}_N^* \to \mathbb{Z}_N^*$ which is very similar to the map sending $a \in \mathbb{Z}_N^*$ to its modular inverse, and the fact that the normalized counting measure on $\mathbb{Z}_N^*$ is invariant under this mapping (cf.\ the discussion before the statement of Lemma \ref{lemma_weightfunctions} below). For this reason, our method does not allow us to replace the uniform average $\frac{1}{\varphi(N)} \sum_{a \in \mathbb{Z}_N^*}$ by some weighted average, say for example with respect to the normalized counting measure on $\{a \in \mathbb{Z}_N^*,~a \leq \eta N \}$ with some fixed $\eta \in (0,1)$. It seems that proving an analogue of our theorems for such weighted averages (following the methods used in this paper) would need significant additional input on the distribution of modular inverses in short intervals, or estimates for Kloosterman sums in short intervals.

\section{A discrepancy estimate for the set of reduced fractions}\label{sec_discrep}

Since we work with reduced fractions, sieve methods naturally play an important role. For the convenience of the reader we recall the fundamental lemma of sieve theory in the formulation of \cite[Theorem 18.11]{kouk}.
\begin{lemma}[Fundamental lemma of sieve theory]\label{lemma_fund}
Let $(a_n)_{n \geq 1}$ be non-negative reals, such that $\sum_{n=1}^\infty a_n < \infty$. Let $\mathcal{P}$ be a finite set of primes, and write $P = \prod_{p \in \mathcal{P}} p$. Set $y = \max \mathcal{P}$, and $A_d = \sum_{n \equiv 0 \mod d} a_n$. Assume that there exist a multiplicative function $g$ such that $0 \le g(p) < p$ for all $p \in \mathcal{P}$, a real number $x$, and positive reals $c_1,c_2$ such that
$$
A_d =: x \frac{g(d)}{d} + r_d, \qquad d \mid P,
$$
and
$$
\prod_{p \in (y_1, y_2] \cap \mathcal{P}} \left( 1 - \frac{g(p)}{p} \right)^{-1} < \left( \frac{\log y_2}{\log y_1} \right)^{c_1} \left(1 + \frac{c_2}{\log y_1} \right), \qquad 3/2 \leq y_1 \leq y_2 \leq y.
$$
For any real $u \ge 1$,
$$
\sum_{\mathrm{gcd} (n,P)=1} a_n = \left( 1 + O ( u^{-u/2} ) \right) x \prod_{p \in \mathcal{P}} \left(1  -\frac{g(p)}{p} \right) + O \left( \sum_{d \leq y^u,~d \mid P} |r_d| \right)
$$
with implied constants depending only on $c_1,c_2$.
\end{lemma}
\noindent We will only use the fundamental lemma with $g=1$. In this case the constants $c_1,c_2$, and consequently the implied constants in the claim are uniform in the sequence $(a_n)_{n \ge 1}$ and the set of primes $\mathcal{P}$.

We now prove a discrepancy estimate for the set of reduced fractions with a fixed denominator $N$. Note that the corresponding result for the set of all (not necessarily reduced) fractions with a fixed denominator $N$ is trivial. In the statement of the lemma, and in the sequel, we write $\lambda$ for the Lebesgue measure. 
\begin{lemma}\label{lemma_Manuel}
Let $N \ge 2$ be an integer, and let $I \subseteq [0,1]$ be an interval. For all reals $T \geq \log N$ and $u \ge 1$,
\[ \sum_{a \in \mathbb{Z}_N^*} \mathds{1}_I \left( \frac{a}{N} \right) = \lambda(I) \varphi(N) \left( 1 + O\left( u^{-u/2} + \frac{\log N}{T} \right) \right) + O\left( T^u \right) \]
with absolute implied constants.
\end{lemma}

\begin{proof} Fix $T \ge \log N$ and $u \ge 1$. Let $N^*$ be the $T$-smooth part of $N$, i.e.\ the largest divisor of $N$ whose prime factors are all at most $T$, and consider
\begin{equation}\label{T_smooth}
\sum_{a \in \mathbb{Z}_N^*} \mathds{1}_I \left( \frac{a}{N} \right) = \sum_{\substack{a=1 \\ \mathrm{gcd} (a,N^*)=1}}^N \mathds{1}_{I} \left( \frac{a}{N} \right) - \sum_{\substack{a=1 \\ \mathrm{gcd} (a,N)=1,~ \mathrm{gcd}(a,N^*)>1}}^N \mathds{1}_{I} \left( \frac{a}{N} \right) .
\end{equation}
We apply the fundamental lemma of sieve theory to the first sum in \eqref{T_smooth}, with $\mathcal{P}$ being the set of prime divisors of $N^*$. Note that $\max \mathcal{P} \le T$ by construction, and that for all $d \mid P$ we have
\[ \sum_{\substack{a=1 \\ d \mid a}}^N \mathds{1}_{I} \left( \frac{a}{N} \right) = \frac{N\lambda(I)}{d} + O(1). \]
We thus obtain
\[ \sum_{\substack{a=1 \\ \mathrm{gcd} (a,N^*)=1}}^N \mathds{1}_{I} \left( \frac{a}{N} \right) = N\lambda(I)\frac{\varphi(N^*)}{N^*}  \left( 1+O \left( u^{-u/2} \right) \right) + O\left( T^u \right) . \]
Observe that
\[ \frac{\varphi(N^*)/N^*}{\varphi(N)/N} = \frac{1}{\prod\limits_{\substack{p \mid N, \\ p>T}}\left( 1-\frac{1}{p} \right)} = \prod_{\substack{p \mid N, \\ p>T}} \left( 1+ \frac{1}{p-1} \right)  = 1 + O \left( \frac{\log N}{T} \right) , \]
which leads to
\[ \sum_{\substack{a=1 \\ \mathrm{gcd} (a,N^*)=1}}^N \mathds{1}_{I} \left( \frac{a}{N} \right) =  \lambda(I) \varphi(N) \left( 1+O \left( u^{-u/2} + \frac{\log N}{T}\right) \right) + O\left( T^u \right) . \]
To estimate the second sum in \eqref{T_smooth}, note that $\frac{N}{N^{*}}$ has $\ll \frac{\log N}{\log T}$ many prime divisors, hence
\[ \begin{split} \sum_{\substack{a=1 \\ \mathrm{gcd} (a,N)=1,~ \mathrm{gcd}(a,N^*)>1}}^N \mathds{1}_{I} \left( \frac{a}{N} \right) & \le \sum_{p \mid \frac{N}{N^*}}~ \sum_{\substack{1 \leq a \leq N, \\ p \mid a}} \mathds{1}_{I} \left( \frac{a}{N} \right) \le \sum_{p \mid \frac{N}{N^*}} \left( \frac{N \lambda (I)}{p} +1 \right) \\ &\ll \frac{\log N}{\log T} \left( \frac{N \lambda (I)}{T} +1 \right)  \ll \frac{\log N}{T} \lambda(I) \varphi (N) + \frac{\log N}{\log T} . \end{split} \]
In the last step we used $N/\log T \ll \varphi (N) \log \log N / \log T \ll \varphi (N)$. The claim of the lemma follows from the previous two formulas.
\end{proof}

The discrepancy estimate in Lemma \ref{lemma_Manuel} gives bounds for $\sum_{a \in \mathbb{Z}_N^*} g(a/N)$ for any function of bounded variation $g$. We will work with functions $g$ supported on a short interval, and use the following variant of Koksma's inequality. The classical form of this inequality \cite[p.\ 143]{kn} corresponds to the case $[c,d]=[0,1]$, i.e.\ without any vanishing condition on $g$. Throughout the rest of the paper, we write $V(g;[0,1])$ for the total variation of a function $g$ on the interval $[0,1]$.
\begin{lemma}\label{lemma_koksma}
Let $x_1, x_2, \ldots, x_M \in [0,1]$, and let $g: [0,1] \to \mathbb{R}$ be a function of bounded variation which vanishes outside an interval $[c,d] \subseteq [0,1]$. Then
\[ \left| \frac{1}{M} \sum_{m=1}^M g (x_m) - \int_0^1 g(x) \, \mathrm{d} x \right| \le V(g;[0,1]) \cdot \sup_{I \subseteq [c,d]} \left| \frac{1}{M} \sum_{m=1}^M \mathds{1}_I (x_m) - \lambda (I) \right| , \]
where the supremum is taken over all intervals $I \subseteq [c,d]$.
\end{lemma}

\begin{proof}
Assume first that $g(0)=0$, and that $g$ is non-decreasing on $[0,d]$. Note that for any $0 \le t \le g(d)$, the set $I(t) = \{ x \in [0,1] \, : \, g(x)>t \}$ is a subinterval of $[c,d]$. A simple application of Fubini's theorem shows that
\begin{eqnarray*} \left| \frac{1}{M} \sum_{m=1}^M g (x_m) - \int_0^1 g(x) \, \mathrm{d} x \right| & = & \left| \int_0^{g(d)} \left( \frac{1}{M} \sum_{m=1}^M \mathds{1}_{I(t)}(x_m) - \lambda (I(t)) \right) \, \mathrm{d} t \right| \\
 & \leq & g(d) \cdot \sup_{I \subseteq [c,d]} \left| \frac{1}{M} \sum_{m=1}^M \mathds{1}_I (x_m) - \lambda (I) \right|,
\end{eqnarray*}
and clearly $g(d) = V(g;[0,1])$. Applying the previous formula to the Jordan decomposition of a general function of bounded variation on $[0,d]$, the claim of the lemma follows under the extra assumption $g(0)=0$. This extra assumption is easily removed as follows. If $[c,d]=[0,1]$, then both the conditions and the claim of the lemma are invariant under adding a constant to $g$, hence we may assume that $g(0)=0$. If $c>0$, then $g(0)=0$ by assumption. If $c=0$ and $d<1$, then we repeat the same arguments for the function $\tilde{g}(x)=g(1-x)$, the point set $\tilde{x}_m = 1-x_m$ and the interval $[\tilde{c},\tilde{d}]=[1-d,1-c]$.
\end{proof}

\section{Weight functions} \label{sec_weight}

Throughout the following sections, we fix integers $N \ge 3$ and $1 \le \eta \le \theta$, and a non-decreasing function $f: \mathbb{N} \to [0,\infty )$. As noted in the Introduction, the continued fraction expansion of $a/N$, $a \in \mathbb{Z}_N^*$ is written as $a/N=[0;a_1,a_2,\ldots, a_r]$, where the partial quotients $a_i=a_i(a/N)$ and the length $r=r(a/N)$ are functions of $a/N$. The convergents to $a/N$ are denoted by $p_i/q_i=[0;a_1,a_2,\ldots, a_i]$, $0 \le i \le r$ (here, $a_i = 1$ is possible), where the convergent numerators $p_i=p_i(a/N)$ and denominators $q_i=q_i(a/N)$ are likewise functions of $a/N$. The main object in the following sections is the restricted sum
\begin{equation} \label{S_def}
 S_{f,\eta,\theta} \left( \frac{a}{N} \right) := \sum_{i=1}^r \mathds{1}_{\{\eta \leq a_i \le \theta \}} f(a_i).
\end{equation}

Given any integers $k \ge 1$, $b \in \mathbb{Z}_k^*$ and $m \ge 1$, we define intervals $I(b/k,m)$ and $I'(b/k,m)$ as follows. Assume first that $k \ge 2$, and consider the continued fraction expansion $b/k=[0;b_1,b_2,\ldots, b_s]$ with $b_s>1$. Let $I(b/k,m)$ be the interval with endpoints $[0;b_1,b_2,\ldots, b_s,m]$ (included if and only if $m>1$) and $[0;b_1,b_2,\ldots, b_s, m+1]$ (not included), and let $I'(b/k,m)$ be the interval with endpoints $[0;b_1,b_2,\ldots, b_{s-1},b_s-1,1,m]$ (included if and only if $m>1$) and $[0;b_1,b_2,\ldots, b_{s-1},b_s-1,1,m+1]$ (not included). This definition cannot be directly used for the (unimportant) special case $k=b=1$, so for completeness in this case we define
\[ I(1,m) = \left\{ \begin{array}{ll} (1/2, 1) & \textrm{if } m=1, \\ \left( [0;m+1], [0;m] \right] & \textrm{if } m>1, \end{array} \right. \]
and
\[ I'(1,m) = \left\{ \begin{array}{ll} \left( [0;1,1], [0;1,2] \right) & \textrm{if } m=1, \\ \left[ [0;1,m], [0;1,m+1] \right) & \textrm{if } m>1. \end{array} \right. \]

We introduce the weight function $w_{f,\eta, \theta} (b/k,x)$, which is defined as
\begin{equation} \label{def_weight} w_{f,\eta,\theta} \left( \frac{b}{k}, x \right) = \sum_{m=\eta}^{\theta} f(m)\cdot\left( \mathds{1}_{I(b/k,m)} \left(x \right) +  \mathds{1}_{I'(b/k,m)} \left(x\right)\right), \qquad x \in [0,1]. \end{equation}
The intervals $I(b/k,m)$ and $I'(b/k,m)$ are constructed in such a way that for any $a \in \mathbb{Z}_N^*$ with continued fraction expansion $a/N=[0;a_1,a_2,\ldots, a_r]$ we have
\[ \sum_{i=1}^r \mathds{1}_{\{ a_i=m \}} = \sum_{k=1}^{N-1} \sum_{b \in \mathbb{Z}_k^*} \left( \mathds{1}_{I(b/k,m)} \left( \frac{a}{N} \right) +  \mathds{1}_{I'(b/k,m)} \left( \frac{a}{N} \right) \right). \]
The right-hand side of this equation shows how $I(b/k,m)$ and $I'(b/k,m)$ are used to count how often the partial quotient $m$ occurs in the continued fraction expansion of $a/N$: occurrences after a partial quotient greater than $1$ are counted with the indicator of $I(b/k,m)$, whereas occurrences after a partial quotient equal to $1$ are counted with the indicator of $I'(b/k,m)$. Multiplying the previous formula by $f(m)$ and summing over $\eta \le m \le \theta$ thus leads to
\[ \sum_{i=1}^r \mathds{1}_{\{ \eta \le a_i \le \theta \}} f(a_i) = \sum_{k=1}^{N-1} \sum_{b \in \mathbb{Z}_k^*} w_{f,\eta,\theta} \left( \frac{b}{k}, \frac{a}{N} \right) . \]
More generally, for any $A \subseteq \{ 1,2,\ldots, N-1 \}$, we have
\begin{equation}\label{weightfunctionidentity}
\sum_{\substack{1 \leq i \leq r, \\ q_{i-1} \in A}} \mathds{1}_{\{ \eta \le a_i \le \theta \}} f(a_i) = \sum_{k \in A} \sum_{b \in \mathbb{Z}_k^*} w_{f,\eta,\theta} \left( \frac{b}{k}, \frac{a}{N} \right) .
\end{equation}

Before we continue with the formal proofs, we give an outline of the general strategy. To study the sum of partial quotients $S(a/N)$, we will use $f(x)=x$. Finding the expected value and estimating the variance of the restricted sums $S_{f,\eta,\theta}$ as defined in \eqref{S_def} for this choice of $f$ and for a wide range of parameters $\eta,\theta$ will allow us to understand the distribution of the unrestricted sum $S(a/N)$ as well. Note that the cutoff at $\theta$ is necessary to obtain meaningful results, since (as mentioned in the Introduction) otherwise the average order of the sum of partial quotients would be dominated by the contribution of very few fractions $a/N$ that possess unusually large partial quotients, thereby disrupting the concentration effect which we are trying to detect.

We will take a sum of $w_{f,\eta, \theta}(b/k,a/N)$ over all $a \in \mathbb{Z}_N^*$, all $b \in \mathbb{Z}_k^*$, and finally over all $k< N$. To actually calculate this sum, we use the regularity of the distribution of the sets of reduced fractions $(a/N)_{a \in \mathbb{Z}_N^*}$ resp.\ $(b/k)_{b \in \mathbb{Z}_k^*}$ as asserted by Lemma \ref{lemma_Manuel}, and the fact that by Lemma \ref{lemma_koksma} a sum of function evaluations can be efficiently approximated by an integral, assuming the regularity of the set of sampling points. To estimate the number of partial quotients $L_{[b,c]}(a/N)$ which fall into the interval $[b,c]$, we use a similar reasoning but with $f(x)=1$, $\eta =b$ and $\theta =c$ instead of $f(x)=x$. We note that some of the basic ideas of this approach are taken from papers of Larcher \cite{larcher} and Rukavishnikova \cite{ruka_1,ruka_2}.

The following two lemmas describe how to estimate the first and second moments of the restricted sum $S_{f,\eta,\theta} (a/N)$ in terms of the weight function $w_{f,\eta, \theta}$. A key point of the lemmas is the fact that the summation over $k$ only runs up to $\sqrt{N}$, and not all the way up to $N$; this is crucial since for $k$ of size near $\sqrt{N}$ the errors coming from the discrepancy bounds in Lemmas \ref{lemma_Manuel} and \ref{lemma_koksma} start being too large. The intuitive explanation why such a reduction for the range of $k$ is possible is that the continued fraction expansion $[0;a_1, \dots, a_r]$ can bijectively be flipped around to give the continued fraction $[0;a_r, \dots, a_1]$ of a reduced fraction with the same denominator. Consequently, when summing over all possible values $a \in \mathbb{Z}_N^*$, every admissible continued fraction expansion occurs once in ``correct'' order and once in reflected order as well, so that instead of taking a statistics of all partial quotients of $a/N$, we can instead only consider those partial quotients associated to convergent denominators going up to $\sqrt{N}$, and then compensate for this reduced range by noting that the remaining partial quotients appear at the initial segment of the reflected continued fraction (so that we get the correct result when restricting to $k < \sqrt{N}$ and compensating for this by multiplying the overall count by a factor of 2). However, the actual situation is slightly more complicated because $[0;a_r, \dots, a_1]$ is not an admissible continued fraction expansion if $a_1=1$, and because we have to be careful that partial quotients which are associated with convergent denominators close to $\sqrt{N}$ are neither omitted nor counted twice when invoking the reflection argument. 

\begin{lemma}\label{lemma_weightfunctions} Assume that $\sqrt{N} > \theta+2$. Then
\[ \begin{split} \sum_{a \in \mathbb{Z}_N^*} S_{f,\eta,\theta} \left( \frac{a}{N} \right)   & = 2 \sum_{1 \le k < \sqrt{N}} \sum_{a \in \mathbb{Z}_N^*} \sum_{b \in \mathbb{Z}_k^*} w_{f,\eta,\theta} \left( \frac{b}{k} , \frac{a}{N} \right) \\ & \quad +O \left( N \sum_{m=\eta}^{\theta} \frac{f(m)}{m^2} + \sum_{\frac{\sqrt{N}}{\theta +2} < k < \sqrt{N}} \sum_{a \in \mathbb{Z}_N^*} \sum_{b \in \mathbb{Z}_k^*} w_{f,\eta,\theta} \left( \frac{b}{k} , \frac{a}{N} \right) \right) \end{split} \]
with an absolute implied constant.
\end{lemma}

\begin{proof} For $a \in \mathbb{Z}_N^*$ with $a \leq N/2$ and continued fraction expansion $a/N=[0;a_1,a_2,\ldots, a_r]$, we define the integer $a^*$ by $a^*/N = [0;a_r, \dots, a_1]$. Note that $a \leq N/2$ implies that $a_1 \geq 2$, so that $[0;a_r,\ldots, a_1]$ indeed is an admissible continued fraction expansion. Note also that $a \mapsto a^*$ is a bijective map from the set $\{a \in \mathbb{Z}_N^* \, : \, a \leq N/2\}$ onto itself. The convergents to $a^*/N$ are $\frac{p_{r-i}(a^*/N)}{q_{r-i}(a^*/N)} = [0;a_r, a_{r-1}, \ldots, a_{i+1}]$, and with the usual convention $q_{-1} := 0$, one readily checks the identity
\begin{equation} \label{identity} 
q_i(a/N) q_{r-i}(a^*/N) + q_{i-1}(a/N) q_{r-i-1}(a^*/N) = N, \qquad 1 \le i \le r . 
\end{equation}
In particular, at least one of the inequalities $q_{i-1}(a/N)<\sqrt{N}$, $q_{r-i}(a^*/N)<\sqrt{N}$ holds, and we can write
\begin{equation} \begin{split} \label{we_can_write} S_{f,\eta,\theta} \left( \frac{a}{N} \right) & = \sum_{\substack{i=1 \\ q_{i-1}(a/N)<\sqrt{N}}}^r \mathds{1}_{\{\eta \leq a_i \le \theta \}} f(a_i) + \sum_{\substack{i=1 \\ q_{r-i}(a^*/N) < \sqrt{N}}}^r \mathds{1}_{\{ \eta \leq a_i \le \theta \}} f(a_i) \\
& \qquad - \sum_{\substack{i=1 \\ q_{i-1}(a/N), q_{r-i}(a^*/N) < \sqrt{N}}}^r \mathds{1}_{\{\eta \leq  a_i \le \theta \}} f(a_i) . \end{split} \end{equation}
Clearly, $q_{r-i}(a^*/N) < \sqrt{N}$ implies that
\[ N < q_i (a/N) \sqrt{N} + q_{i-1}(a/N) \sqrt{N} \le (a_i+2) q_{i-1}(a/N) \sqrt{N} , \]
hence
\begin{equation} \label{we_can_write_2} \sum_{\substack{i=1 \\ q_{i-1}(a/N), q_{r-i}(a^*/N) < \sqrt{N}}}^r \mathds{1}_{\{\eta \leq  a_i \le \theta \}} f(a_i) \le \sum_{\substack{i=1 \\ \frac{\sqrt{N}}{\theta +2} < q_{i-1}(a/N) < \sqrt{N}}}^r \mathds{1}_{\{\eta \leq a_i \le \theta \}} f(a_i) . \end{equation}
Thus we have 
\begin{equation} \label{thus_we_have}
 S_{f,\eta,\theta} \left( \frac{a}{N} \right) \leq \sum_{\substack{i=1 \\ q_{i-1}(a/N)<\sqrt{N}}}^r \mathds{1}_{\{\eta \leq a_i \le \theta \}} f(a_i) + \sum_{\substack{i=1 \\ q_{r-i}(a^*/N) < \sqrt{N}}}^r \mathds{1}_{\{\eta \leq a_i \le \theta \}} f(a_i)
\end{equation}
and
\begin{equation} \begin{split}  \label{thus_we_have_2}
S_{f,\eta,\theta} \left( \frac{a}{N} \right) & \geq \sum_{\substack{i=1 \\ q_{i-1}(a/N)<\sqrt{N}}}^r \mathds{1}_{\{\eta \leq a_i \le \theta \}} f(a_i) + \sum_{\substack{i=1 \\ q_{r-i}(a^*/N) < \sqrt{N}}}^r \mathds{1}_{\{\eta \leq a_i \le \theta \}} f(a_i) \\
& \qquad -\sum_{\substack{i=1 \\ \frac{\sqrt{N}}{\theta +2} < q_{i-1}(a/N) < \sqrt{N}}}^r \mathds{1}_{\{\eta \leq a_i \le \theta \}} f(a_i). \end{split}
\end{equation} 
Writing $(a_i^*)_{1 \leq i \leq r}$ for the partial quotients of $a^*/N$, by re-indexing the sum by $i \mapsto r-i+1$ we have
\begin{equation}\label{symmetry}
\sum_{\substack{i=1 \\ q_{r-i}(a^*/N) < \sqrt{N}}}^r \mathds{1}_{\{\eta \leq a_i \le \theta \}} f(a_i) = \sum_{\substack{i=1 \\ q_{i-1}(a^*/N) < \sqrt{N}}}^r \mathds{1}_{\{\eta \leq a_i^* \le \theta \}} f(a_i^*).
\end{equation}
We thus obtain
\begin{equation} \label{thus_obtain} \begin{split} S_{f,\eta,\theta} \left( \frac{a}{N} \right) &\le \sum_{\substack{i=1 \\ q_{i-1}(a/N)<\sqrt{N}}}^r \mathds{1}_{\{ \eta \leq  a_i \le \theta \}} f(a_i) + \sum_{\substack{i=1 \\ q_{i-1}(a^*/N) < \sqrt{N}}}^r \mathds{1}_{\{\eta \leq a_i^* \le \theta \}} f(a_i^*) \end{split} \end{equation}
and
\begin{equation} \begin{split} \label{thus_obtain_2} 
S_{f,\eta,\theta} \left( \frac{a}{N} \right) &\ge \sum_{\substack{i=1 \\ q_{i-1}(a/N)<\sqrt{N}}}^r \mathds{1}_{\{\eta \leq a_i \le \theta \}} f(a_i) + \sum_{\substack{i=1 \\ q_{i-1}(a^*/N) < \sqrt{N}}}^r \mathds{1}_{\{\eta \leq a_i^* \le \theta \}} f(a_i^*) \\
& \qquad - \sum_{\substack{i=1 \\ \frac{\sqrt{N}}{\theta +2} < q_{i-1}(a/N) < \sqrt{N}}}^r \mathds{1}_{\{\eta \leq a_i \le \theta \}} f(a_i). \end{split} \end{equation}
Observe that the first and second sum on the right-hand side of \eqref{thus_obtain} and \eqref{thus_obtain_2} are in each case the same function applied to the fraction $a/N$ resp.\ $a^*/N$. Since $a \mapsto a^*$ is a bijective map from the set $\{a \in \mathbb{Z}_N^* \, : \, a \leq N/2\}$ onto itself, summing over $a$ leads to
\begin{equation}\label{Sfupperlowerbound}
\sum_{\substack{a \in \mathbb{Z}_N^*,\\a \leq N/2}} S_{f,\eta,\theta} \left( \frac{a}{N} \right) \le 2 \sum_{\substack{a \in \mathbb{Z}_N^*,\\a \leq N/2}}~ \sum_{\substack{i=1 \\ q_{i-1}<\sqrt{N}}}^r \mathds{1}_{\{\eta \leq a_i \le \theta \}} f(a_i)
\end{equation}
and
\begin{equation} \label{Sfupperlowerbound2}  \begin{split}
\sum_{\substack{a \in \mathbb{Z}_N^*,\\a \leq N/2}} S_{f,\eta,\theta} \left( \frac{a}{N} \right) & \ge  2 \sum_{\substack{a \in \mathbb{Z}_N^*,\\a \leq N/2}} ~\sum_{\substack{i=1 \\ q_{i-1}<\sqrt{N}}}^r \mathds{1}_{\{\eta \leq a_i \le \theta \}} f(a_i) \\
& \quad  -  \sum_{\substack{a \in \mathbb{Z}_N^*, \\ a \le N/2}} \sum_{\substack{i=1 \\ \frac{\sqrt{N}}{\theta +2} < q_{i-1} < \sqrt{N}}}^r \mathds{1}_{\{\eta \leq a_i \le \theta \}} f(a_i). \end{split}
\end{equation}

Consider now $a \in \mathbb{Z}_N^*$ with $a > N/2$ and with continued fraction expansion $a/N=[0;a_1,a_2,\ldots, a_r]$. In this case we have $a_1=1$, and we define $a^*$ by 
$$
\frac{a^*}{N} = [0;1,a_r-1, a_{r-1},a_{r-2}, \dots, a_3, a_2+1].
$$
Then $a \mapsto a^*$ is a bijective map from the set $\{a \in \mathbb{Z}_N^* \, : \, a > N/2\}$ onto itself. It is easy to see that the denominators of the convergents to $[0;1,a_r-1, a_{r-1}, \dots, a_3, a_2+1]$ are the same as those of $[0;a_r, a_{r-1},\dots, a_3, a_2+1]$, with a shift of the index by one position. Consequently, in the present case instead of \eqref{identity} we now have
$$
q_i(a/N) q_{r-i+1}(a^*/N) + q_{i-1}(a/N) q_{r-i}(a^*/N) = N, \qquad 2 \le i \le r-1,
$$
and at least one of the inequalities $q_{i-1}(a/N) <\sqrt{N}, q_{r-i+1}(a^*/N) < \sqrt{N}$ is true for all $1 \le i \le r$. We obtain versions of \eqref{we_can_write}--\eqref{thus_we_have_2}, where the index $r-i$ of $q_{r-i}(a^*/N)$ is now replaced by $r-i+1$.

In the case $a \leq N/2$ we had the simple symmetry relation $a_i^* = a_{r-i+1}$, $1 \le i \le r$, which gave \eqref{symmetry}. Now we have the slightly more complicated relations
$$
a_1^* = 1, \qquad a_2^* = a_r - 1, \qquad a^*_r = a_2 +1, 
$$
and
$$
a_{i}^* = a_{r-i+2}, \qquad 3 \leq i \leq r-1.
$$
Thus instead of \eqref{symmetry} we now have, when re-indexing the sum by $i \mapsto r-i+2$, 
\begin{equation*} \begin{split} \sum_{\substack{i=1 \\ q_{r-i+1}(a^*/N) < \sqrt{N}}}^r \mathds{1}_{\{\eta \leq a_i \le \theta \}} f(a_i) & = \sum_{\substack{i=2 \\ q_{i-1}(a^*/N) < \sqrt{N}}}^r \mathds{1}_{\{\eta \leq a_{r-i+2} \le \theta \}} f(a_{r-i+2}) \\ & = \mathds{1}_{\{\eta \leq a_2^* +1 \le \theta \}} f(a_2^* + 1) 
+  \sum_{\substack{i=3 \\ q_{i-1}(a^*/N) < \sqrt{N}}}^{r-1} \mathds{1}_{\{\eta \leq a_i^* \le \theta \}} f(a_i^*) . \end{split} \end{equation*}
Here we used the fact that the term for $i=1$ on the left-hand side can be omitted, since $q_r (a^*/N) < \sqrt{N}$ is impossible. We also used the fact that the term for $i=r$ on the right-hand side is actually zero, since $q_{r-1}(a^*/N) < \sqrt{N}$ and $a_2 \le \theta$ cannot be both satisfied. Indeed, if both these conditions were true, then this would imply
\[ N = q_r(a^*/N) \leq (a_2 + 2) q_{r-1}(a^*/N) \le (\theta+2) \sqrt{N}, \]
in contradiction to our assumption that $\sqrt{N}>\theta +2$. 

Consequently, for all $a \in \mathbb{Z}_N^*$, $a > N/2$,
\begin{equation}\label{thus_obtain_B}
\begin{split} & S_{f,\eta,\theta} \left( \frac{a}{N} \right) \\
\le &\sum_{\substack{i=1 \\ q_{i-1}(a/N)<\sqrt{N}}}^r \mathds{1}_{\{ \eta \leq  a_i \le \theta \}} f(a_i) + \mathds{1}_{\{\eta \leq a_2^* +1 \le \theta \}} f(a_2^* + 1) + \sum_{\substack{i=3 \\ q_{i-1}(a^*/N) < \sqrt{N}}}^{r-1} \mathds{1}_{\{\eta \leq a_i^* \le \theta \}} f(a_i^*) \\
= &\sum_{\substack{i=1 \\ q_{i-1}(a/N)<\sqrt{N}}}^r \mathds{1}_{\{ \eta \leq  a_i \le \theta \}} f(a_i) + \sum_{\substack{i=2 \\ q_{i-1}(a^*/N) < \sqrt{N}}}^{r-1} \mathds{1}_{\{\eta \leq a_i^* \le \theta \}} f(a_i^*) \\ & \quad + \mathds{1}_{\{ a_2^* = \eta -1 \}} f(\eta) + \mathds{1}_{\{ \eta \le a_2^* \le \theta -1 \}} (f(a_2^*+1)-f(a_2^*)) - \mathds{1}_{\{ a_2^* = \theta \}} f(\theta) . \end{split}
\end{equation}
We now sum the previous formula over all $a \in \mathbb{Z}_N^*$, $a>N/2$. Since $a_1^*=1$, having $a_2^*=m$ means that $a^*/N$ lies in the interval $[\frac{1}{1+1/m}, \frac{1}{1+1/(m+1)}]$ of length $\frac{1}{(m+1)(m+2)}$. The number of such $a \in \mathbb{Z}_N^*$ is at most $\frac{N}{(m+1)(m+2)}+1 \le \frac{2N}{(m+1)(m+2)}$. Hence
\[ \begin{split} \sum_{\substack{a \in \mathbb{Z}_N^*,\\ a>N/2}} \big( \mathds{1}_{\{ a_2^* = \eta -1 \}} f(\eta) + &\mathds{1}_{\{ \eta \le a_2^* \le \theta -1 \}} (f(a_2^*+1)-f(a_2^*)) - \mathds{1}_{\{ a_2^* = \theta \}} f(\theta) \big) \\ &\le 2 N \left( \frac{f(\eta)}{\eta (\eta +1)} + \sum_{\eta \le m \le \theta -1} \frac{f(m+1)-f(m)}{(m+1)(m+2)} \right) \le 2 N \sum_{m=\eta}^{\theta} \frac{f(m)}{m^2}, \end{split} \]
and \eqref{thus_obtain_B} leads to
\begin{equation}\label{Sfupperlowerbound_B}
\sum_{\substack{a \in \mathbb{Z}_N^*,\\a > N/2}} S_{f,\eta,\theta} \left( \frac{a}{N} \right) \le 2 \sum_{\substack{a \in \mathbb{Z}_N^*,\\a > N/2}} ~\sum_{\substack{i=1 \\ q_{i-1}<\sqrt{N}}}^r \mathds{1}_{\{\eta \leq  a_i \le \theta \}} f(a_i) + 2 N \sum_{m=\eta}^{\theta} \frac{f(m)}{m^2} .
\end{equation}

As a lower bound, by re-indexing we similarly obtain
$$
\sum_{\substack{i=1 \\ q_{r-i+1}(a^*/N) < \sqrt{N}}}^r \mathds{1}_{\{\eta \leq a_i \le \theta \}} f(a_i) \geq \sum_{\substack{i=1 \\ 1 < q_{i-1}(a^*/N) < \sqrt{N}}}^r \mathds{1}_{\{\eta \leq a_i^* \le \theta \}} f(a_i^*).
$$
Here we used the fact that the terms for $i=1,2,r$ on the right-hand side are zero, the latter as a consequence of our assumption $\sqrt{N}>\theta + 2$. Consequently, summing over all $a \in \mathbb{Z}_N^*$, $a > N/2$ yields
\begin{equation} \label{Sfupperlowerbound_B_2} \begin{split}
\sum_{\substack{a \in \mathbb{Z}_N^*,\\a > N/2}} S_{f,\eta,\theta} \left( \frac{a}{N} \right)  & \ge 2 \sum_{\substack{a \in \mathbb{Z}_N^*,\\a > N/2}}~ \sum_{\substack{i=1 \\ 1 < q_{i-1}<\sqrt{N}}}^r \mathds{1}_{\{\eta \leq a_i \le \theta \}} f(a_i) \\ & \quad  -  \sum_{\substack{a \in \mathbb{Z}_N^*, \\ a>N/2}} \sum_{\substack{i=1 \\ \frac{\sqrt{N}}{\theta +2} < q_{i-1} < \sqrt{N}}}^r \mathds{1}_{\{\eta \leq a_i \le \theta \}} f(a_i). \end{split}
\end{equation} 

Combining \eqref{Sfupperlowerbound}, \eqref{Sfupperlowerbound2}, \eqref{Sfupperlowerbound_B}, \eqref{Sfupperlowerbound_B_2} and the identity \eqref{weightfunctionidentity}, we finally obtain
\begin{equation}\label{upper_weightfct} \begin{split} \sum_{\substack{a \in \mathbb{Z}_N^*}} S_{f,\eta,\theta} \left( \frac{a}{N} \right) &\le 2 \sum_{\substack{a \in \mathbb{Z}_N^*}} \sum_{\substack{i=1 \\ q_{i-1}<\sqrt{N}}}^r \mathds{1}_{\{\eta \leq a_i \le \theta \}} f(a_i) + 2N \sum_{m=\eta}^{\theta} \frac{f(m)}{m^2} \\ &= 2 \sum_{1 \le k < \sqrt{N}} \sum_{a \in \mathbb{Z}_N^*} \sum_{b \in \mathbb{Z}_k^*} w_{f,\eta, \theta} \left( \frac{b}{k}, \frac{a}{N} \right) + 2N \sum_{m=\eta}^{\theta} \frac{f(m)}{m^2} \end{split} \end{equation}
and
\[ \begin{split} \sum_{\substack{a \in \mathbb{Z}_N^*}} S_{f,\eta,\theta} \left( \frac{a}{N} \right) &\ge 2 \sum_{\substack{a \in \mathbb{Z}_N^*}} \sum_{\substack{i=1 \\1 < q_{i-1}<\sqrt{N}}}^r \mathds{1}_{\{\eta \leq a_i \le \theta \}} f(a_i) -  \sum_{a \in \mathbb{Z}_N^*} \sum_{\substack{i=1 \\ \frac{\sqrt{N}}{\theta +2} < q_{i-1} < \sqrt{N}}}^r \mathds{1}_{\{\eta \leq a_i \le \theta \}} f(a_i) \\ &= 2 \sum_{1<k<\sqrt{N}} \sum_{a \in \mathbb{Z}_N^*} \sum_{b \in \mathbb{Z}_k^*} w_{f,\eta, \theta} \left( \frac{b}{k}, \frac{a}{N} \right) - \sum_{\frac{\sqrt{N}}{\theta +2}<k<\sqrt{N}} \sum_{a \in \mathbb{Z}_N^*} \sum_{b \in \mathbb{Z}_k^*} w_{f,\eta, \theta} \left( \frac{b}{k}, \frac{a}{N} \right) . \end{split} \]
It remains to show that the term for $k=1$ is negligible. This is easily seen from the fact that there are $\ll N/m^2$ fractions $a/N$ which fall into $I(1,m)$ or $I'(1,m)$, and thus
\begin{equation}\label{k=1_estimate}  \sum_{a \in \mathbb{Z}_N^*} w_{f,\eta, \theta} \left( 1,\frac{a}{N} \right) \ll N \sum_{m=\eta}^{\theta} \frac{f(m)}{m^2} . \end{equation}
\end{proof}

\begin{lemma} \label{lemma_weight_square} Assume that $\sqrt{N} > \theta+2$. For any integer $1 \leq R \le \sqrt{N}$,
\[ \begin{split} \sum_{a \in \mathbb{Z}_N^*} S_{f,\eta,\theta} \left( \frac{a}{N} \right)^2 & \leq 8 \sum_{1 \le k < \sqrt{N}/R} \sum_{a \in \mathbb{Z}_N^*} \sum_{b \in \mathbb{Z}_k^*} w_{f,\eta,\theta} \left( \frac{b}{k} , \frac{a}{N} \right) S_{f,\eta,\theta} \left( \frac{b}{k} \right)  \\
& \quad +12 \sum_{a \in \mathbb{Z}_N^*} S_{f^2,\eta,\theta} \left( \frac{a}{N} \right) \\
& \quad +8 \left( \sum_{a \in \mathbb{Z}_N^*} S_{f,\eta,\theta} \left(\frac{a}{N} \right)^2 \right)^{1/2}  \left(\sum_{\sqrt{N}/R  \leq k <\sqrt{N}} \sum_{a \in \mathbb{Z}_N^*} \sum_{b \in \mathbb{Z}_k^*} w_{f^2,\eta,\theta} \left( \frac{b}{k} , \frac{a}{N} \right) \right)^{1/2} \\
& \quad +24 f(R) (1+\log R) \sum_{a \in  \mathbb{Z}_N^*} S_{f,\eta,\theta} \left( \frac{a}{N} \right) \\
& \quad + 4 \left( \sum_{a \in \mathbb{Z}_N^*} S_{f,\eta,\theta} \left( \frac{a}{N} \right)^2 \right)^{1/2} \Xi_{f,\eta, \theta}^{1/2} N^{1/2} + \Xi_{f,\eta, \theta} N
\end{split} \]
with
\begin{equation}\label{Xidef}
\Xi_{f,\eta, \theta} = \frac{2 f(\eta)^2}{\eta (\eta +1)} + \sum_{\eta \le m \le \theta -1} \frac{2 (f(m+1)-f(m))^2}{(m+1)(m+2)} + \frac{2 f(\theta)^2}{(\theta +1)(\theta +2)} .
\end{equation}
\end{lemma}

\begin{proof} Given $a \in \mathbb{Z}_N^*$, let $a^*$ and $a_i^*$, $1 \le i \le r$ be as in the proof of Lemma \ref{lemma_weightfunctions}; recall that the definition depends on whether $a \le N/2$ or $a>N/2$. By \eqref{thus_obtain} and \eqref{thus_obtain_B}, for all $a \in \mathbb{Z}_N^*$ we have
\[ S_{f,\eta,\theta} \left( \frac{a}{N} \right) \le \sum_{\substack{i=1 \\ q_{i-1}(a/N)<\sqrt{N}}}^r \mathds{1}_{\{\eta \leq a_i \le \theta \}} f(a_i) + \sum_{\substack{i=1 \\ q_{i-1}(a^*/N) < \sqrt{N}}}^r \mathds{1}_{\{\eta \leq a_i^* \le \theta \}} f(a_i^*) + \xi_{f,\eta, \theta} \left( \frac{a}{N} \right) \]
with
\[ \xi_{f,\eta, \theta} \left( \frac{a}{N} \right) := \mathds{1}_{\{ a>N/2 \}} \left( \mathds{1}_{\{ a_2^* =\eta-1 \}} f(\eta) + \mathds{1}_{\{ \eta \le a_2^* \le \theta -1 \}} (f(a_2^*+1) - f(a_2^*)) - \mathds{1}_{\{ a_2^* =\theta \}} f(\theta) \right) . \]
Hence
\[ \begin{split} S_{f,\eta,\theta} \left( \frac{a}{N} \right)^2 & \leq  2 \Bigg( \sum_{\substack{i=1 \\ q_{i-1}(a/N)<\sqrt{N}}}^r \mathds{1}_{\{\eta \leq a_i \le \theta \}} f(a_i) \Bigg)^2 + 2 \Bigg( \sum_{\substack{i=1 \\ q_{i-1}(a^*/N) < \sqrt{N}}}^r \mathds{1}_{\{\eta \leq a_i^* \le \theta \}} f(a_i^*) \Bigg)^2  \\
& \quad + 2 S_{f,\eta,\theta} \left( \frac{a}{N} \right) \xi_{f,\eta, \theta} \left( \frac{a}{N} \right) + 2 S_{f,\eta,\theta} \left( \frac{a^*}{N} \right) \xi_{f,\eta, \theta} \left( \frac{a}{N} \right) +  \xi_{f,\eta, \theta} \left( \frac{a}{N} \right)^2 , \end{split} \]
and after summing over all $a \in \mathbb{Z}_N^*$,
\begin{equation}\label{Ssquareupperbound}
\begin{split} \sum_{a \in \mathbb{Z}_N^*} S_{f,\eta,\theta} \left( \frac{a}{N} \right)^2 & \le  4 \sum_{a \in \mathbb{Z}_N^*} \Bigg( \sum_{\substack{i=1 \\ q_{i-1}(a/N)<\sqrt{N}}}^r \mathds{1}_{\{\eta \leq a_i \le \theta \}} f(a_i) \Bigg)^2 \\ & \quad + 2 \sum_{a \in \mathbb{Z}_N^*} S_{f,\eta,\theta} \left( \frac{a}{N} \right) \xi_{f,\eta, \theta} \left( \frac{a}{N} \right) \\
& \quad +2 \sum_{a \in \mathbb{Z}_N^*} S_{f,\eta,\theta} \left( \frac{a^*}{N} \right) \xi_{f,\eta, \theta} \left( \frac{a}{N} \right) + \sum_{a \in \mathbb{Z}_N^*} \xi_{f,\eta, \theta} \left( \frac{a}{N} \right)^2 . \end{split}
\end{equation}

We first estimate the last three sums on the right-hand side of \eqref{Ssquareupperbound}. Since
\[  \xi_{f,\eta, \theta} \left( \frac{a}{N} \right)^2 = \mathds{1}_{\{ a>N/2 \}} \left( \mathds{1}_{\{ a_2^* =\eta-1 \}} f(\eta)^2 + \mathds{1}_{\{ \eta \le a_2^* \le \theta -1 \}} (f(a_2^*+1) - f(a_2^*))^2 + \mathds{1}_{\{ a_2^* =\theta \}} f(\theta)^2 \right) , \]
following the steps in the proof of Lemma \ref{lemma_weightfunctions} leads to
\[ \sum_{a \in \mathbb{Z}_N^*} \xi_{f,\eta, \theta} \left( \frac{a}{N} \right)^2 \le 2N \left( \frac{f(\eta)^2}{\eta (\eta +1)} + \sum_{\eta \le m \le \theta -1} \frac{(f(m+1)-f(m))^2}{(m+1)(m+2)} + \frac{f(\theta)^2}{(\theta +1)(\theta +2)} \right), \]
where by \eqref{Xidef} the right-hand side is $\Xi_{f,\eta, \theta} N$. Using the Cauchy--Schwarz inequality we obtain
\[ 2 \sum_{a \in \mathbb{Z}_N^*} S_{f,\eta,\theta} \left( \frac{a}{N} \right) \xi_{f,\eta, \theta} \left( \frac{a}{N} \right) \le 2 \left( \sum_{a \in \mathbb{Z}_N^*} S_{f,\eta,\theta} \left( \frac{a}{N} \right)^2 \right)^{1/2} \Xi_{f,\eta, \theta}^{1/2} N^{1/2}, \]
and since $a \mapsto a^*$ is a bijection of $\mathbb{Z}_N^*$ the same estimate holds if $S_{f,\eta,\theta} \left( \frac{a}{N} \right)$ is replaced by $S_{f,\eta,\theta} \left( \frac{a^*}{N} \right)$.

It remains to estimate the first sum on the right-hand side of \eqref{Ssquareupperbound}. Fix $a \in \mathbb{Z}_N^*$. The diagonal terms are easily estimated as
\[ \sum_{\substack{i=1 \\ q_{i-1}< \sqrt{N}}}^r \mathds{1}_{ \{\eta \leq a_i \le \theta \}} f(a_i)^2 \le S_{f^2, \eta, \theta} \left( \frac{a}{N} \right) . \]
Let us decompose the off-diagonal terms as
\begin{equation}\label{offdiag-decomposition}
\begin{split} \sum_{\substack{1 \le j<i \le r \\ q_{i-1}<\sqrt{N}}} \mathds{1}_{\{\eta \leq a_j \le \theta \}} f(a_j) &\mathds{1}_{\{\eta \leq a_i \le \theta \}} f(a_i) = \\  &\sum_{\substack{1 \le j<i \le r \\ q_{i-1}<\sqrt{N}/R}} \mathds{1}_{\{\eta \leq a_j \le \theta \}} f(a_j) \mathds{1}_{\{\eta \leq a_i \le \theta \}} f(a_i) \\ &+ \sum_{\substack{1 \le j<i \le r \\ \sqrt{N}/R \leq q_{i-1}<\sqrt{N}}} \mathds{1}_{\{\eta \leq a_j \le \theta \}} f(a_j) \mathds{1}_{\{\eta \le a_i \le \theta \}} \mathds{1}_{\{ a_i \le R \}} f(a_i) \\ &+ \sum_{\substack{1 \le j<i \le r \\ \sqrt{N}/R \leq q_{i-1}<\sqrt{N}}} \mathds{1}_{\{\eta \leq a_j \le \theta \}} f(a_j) \mathds{1}_{\{\eta \le a_i \le \theta \}} \mathds{1}_{\{ a_i>R \}} f(a_i) . \end{split}
\end{equation}
Consider the first sum on the right-hand side of \eqref{offdiag-decomposition}. For any $2 \le i \le r$, the canonical continued fraction expansion of the convergent $p_{i-1}/q_{i-1}$ is $[0;a_1,a_2,\ldots, a_{i-1}]$ if $a_{i-1}>1$, and $[0;a_1,a_2,\ldots, a_{i-2}+1]$ if $a_{i-1}=1$. Therefore
\begin{equation}\label{almost_Sbk} \sum_{1 \le j<i} \mathds{1}_{\{ \eta \le a_j \le \theta \}} f(a_j) \le S_{f,\eta,\theta} \left( \frac{p_{i-1}}{q_{i-1}} \right) + f(1) + \mathds{1}_{\{ a_{i-2}=\theta \}} f(\theta ) , \end{equation}
and using Cauchy--Schwarz and the fact that $\sum_{i=2}^r \mathds{1}_{\{ a_{i-2} = \theta \}} f(\theta )^2 \le S_{f^2, \eta, \theta} (a/N)$, we obtain
\[ \begin{split} \sum_{\substack{1 \le j<i \le r \\ q_{i-1}<\sqrt{N}/R}} \mathds{1}_{\{\eta \leq a_j \le \theta \}} f(a_j) \mathds{1}_{\{\eta \leq a_i \le \theta \}} f(a_i) \le & \sum_{\substack{2 \le i \le r \\ q_{i-1}<\sqrt{N}/R}} \mathds{1}_{\{ \eta \le a_i \le \theta \}} f(a_i) S_{f,\eta, \theta} \left( \frac{p_{i-1}}{q_{i-1}} \right) \\ &+ f(1) S_{f,\eta, \theta} \left( \frac{a}{N} \right) + S_{f^2,\eta, \theta} \left( \frac{a}{N} \right) \\ \le &\sum_{1 \le k < \sqrt{N}/R} \sum_{b \in \mathbb{Z}_k^*} w_{f,\eta,\theta} \left( \frac{b}{k} , \frac{a}{N} \right) S_{f,\eta, \theta} \left( \frac{b}{k} \right) \\&+f(1) S_{f,\eta, \theta} \left( \frac{a}{N} \right) + S_{f^2,\eta, \theta} \left( \frac{a}{N} \right) . \end{split} \]
The second sum on the right-hand side of \eqref{offdiag-decomposition} satisfies
\begin{equation}\label{smallai} \begin{split} 
& \sum_{\substack{1 \le j<i \le r \\ \sqrt{N}/R \leq q_{i-1}<\sqrt{N}}} \mathds{1}_{\{\eta \leq a_j \le \theta \}} f(a_j) \mathds{1}_{\{ \eta \le a_i \le \theta \}} \mathds{1}_{\{ a_i \le R \}} f(a_i)  \\ & \leq f(R) S_{f,\eta,\theta} \left( \frac{a}{N} \right) \sum_{\substack{2 \le i \le r \\ \sqrt{N}/R \leq q_{i-1}<\sqrt{N}}} 1 \\
& \leq f(R) (2+3 \log R) S_{f,\eta,\theta} \left( \frac{a}{N} \right),
\end{split}\end{equation}
where the last estimate follows from $q_{i+2} \geq q_{i+1} + q_{i} \geq 2 q_i$ for all $i$, together with $2/\log 2 \leq 3$. The third sum on the right-hand side of \eqref{offdiag-decomposition} similarly satisfies
\[ \begin{split} \sum_{\substack{1 \le j<i \le r \\ \sqrt{N}/R \leq q_{i-1}<\sqrt{N}}} \mathds{1}_{\{\eta \leq a_j \le \theta \}} f(a_j) &\mathds{1}_{\{ \eta \le a_i \le \theta \}} \mathds{1}_{\{ a_i>R \}} f(a_i) \\ &\le S_{f,\eta,\theta} \left( \frac{a}{N} \right) \sum_{\substack{2 \le i \le r \\ \sqrt{N}/R \le q_{i-1}<\sqrt{N}}} \mathds{1}_{\{ \eta \le a_i \le \theta \}} \mathds{1}_{\{ a_i>R \}} f(a_i) , \end{split} \]
where the crucial property is that the remaining sum has at most one term. The previous three formulas show that
\[ \begin{split} \Bigg( \sum_{\substack{i=1 \\ q_{i-1}<\sqrt{N}}}^r \mathds{1}_{\{\eta \leq a_i \le \theta \}} f(a_i) \Bigg)^2 & \le 2\sum_{1 \le k < \sqrt{N}/R} \sum_{b \in \mathbb{Z}_k^*} w_{f,\eta,\theta} \left( \frac{b}{k} , \frac{a}{N} \right) S_{f,\eta, \theta} \left( \frac{b}{k} \right) \\ & \quad + 3S_{f^2, \eta, \theta} \left( \frac{a}{N} \right) + 6 f(R) (1+ \log R) S_{f,\eta, \theta} \left( \frac{a}{N} \right) \\ & \quad +2S_{f,\eta,\theta} \left( \frac{a}{N} \right) \sum_{\substack{2 \le i \le r \\ \sqrt{N}/R \le q_{i-1}<\sqrt{N}}} \mathds{1}_{\{ \eta \le a_i \le \theta \}} \mathds{1}_{\{ a_i>R \}} f(a_i) . \end{split} \]
The Cauchy--Schwarz inequality together with the fact that
\[ \begin{split} \Bigg( \sum_{\substack{2 \le i \le r \\ \sqrt{N}/R \le q_{i-1}<\sqrt{N}}} \mathds{1}_{\{ \eta \le a_i \le \theta \}} \mathds{1}_{\{ a_i>R \}} f(a_i) \Bigg)^2 &= \sum_{\substack{2 \le i \le r \\ \sqrt{N}/R \le q_{i-1}<\sqrt{N}}} \mathds{1}_{\{ \eta \le a_i \le \theta \}} \mathds{1}_{\{ a_i>R \}} f(a_i)^2 \\ &\le \sum_{\sqrt{N}/R \le k<\sqrt{N}} \sum_{b \in \mathbb{Z}_k^*} w_{f^2,\eta,\theta} \left( \frac{b}{k} , \frac{a}{N} \right) \end{split} \]
immediately yield the desired upper bound for the first sum in \eqref{Ssquareupperbound}.
\end{proof}

\section{Expected value} \label{sec_expect}

In this section we use Lemma \ref{lemma_weightfunctions} to determine the expected value of $S_{f,\eta,\theta}$.
\begin{thm}\label{expectedvaluetheorem} Let $N \ge 3$ and $1 \le \eta \le \theta$ be integers, and let $f: \mathbb{N} \to [0,\infty )$ be a non-decreasing function with $f(\theta) > 0$. Let
\[ A_{f,\eta, \theta} = \frac{12}{\pi^2} \sum_{m=\eta}^{\theta} f(m) \log \left( 1 +\frac{1}{m(m+2)} \right), \quad B_{f,\eta,\theta}= A_{f,\eta,\theta} + \frac{f(\theta)}{\theta}, \quad C_{f,\eta,\theta}= \frac{5 f(\theta)}{B_{f,\eta,\theta}} . \]
Then
\begin{equation}\label{expectedupperbound}
\frac{1}{\varphi(N)} \sum_{a \in \mathbb{Z}_N^*} S_{f,\eta,\theta} \left( \frac{a}{N} \right) \ll B_{f,\eta,\theta} \log N \log \log N
\end{equation}
with an absolute implied constant. Assuming that $\sqrt{N} > \theta +2$, we also have
\begin{equation}\label{expectedasymptotics} \begin{split}
& \frac{1}{\varphi (N)} \sum_{a \in \mathbb{Z}_N^*} S_{f,\eta,\theta} \left( \frac{a}{N} \right) \\ & = A_{f,\eta,\theta} \log N + O \left( B_{f,\eta,\theta} \left( \frac{(\log (C_{f,\eta,\theta}\log N))^2}{\log \log (C_{f,\eta,\theta}\log N)} + \log \theta \right) \log \log N \right) \end{split}
\end{equation}
with an absolute implied constant.
\end{thm}
\noindent Note that $C_{f,\eta,\theta} \le 5 \theta$. In particular, if $\theta \le (\log N)^C$ with some constant $C>0$, then
\[ \frac{1}{\varphi (N)} \sum_{a \in \mathbb{Z}_N^*} S_{f,\eta,\theta} \left( \frac{a}{N} \right) = A_{f,\eta,\theta} \log N + O \left( B_{f,\eta,\theta} (\log \log N)^3 \right) \]
with an implied constant depending only on $C$. We give the proof of Theorem \ref{expectedvaluetheorem} after a preparatory lemma.
\begin{lemma}\label{integrallemma} For any integer $k \ge 1$,
\[ \sum_{b \in \mathbb{Z}_k^*} \int_0^1 w_{f,\eta,\theta} \left( \frac{b}{k}, x \right) \, \mathrm{d} x = \frac{2\varphi (k)}{k^2} \sum_{m=\eta}^{\theta} f(m) \log \left( 1+\frac{1}{m(m+2)} \right) + O \left( \frac{\sum_{m=\eta}^{\theta}f(m)/m^2}{k(\log (k+1))^{100}} \right) \]
with an absolute implied constant.
\end{lemma}

\begin{proof} The proof for $k=1$ is trivial, thus we may assume that $k \ge 2$. Fix $b \in \mathbb{Z}_k^*$, and consider the continued fraction expansion $b/k=[0;b_1,b_2,\ldots, b_s]$ with $b_s>1$. Let $p_i/q_i=[0;b_1,b_2,\ldots, b_i]$ denote the convergents. Recall the identity $p_s q_{s-1} - p_{s-1} q_s=(-1)^{s+1}$. We have
\[ \int_0^1 w_{f,\eta, \theta} \left( \frac{b}{k}, x \right) \, \mathrm{d}x = \sum_{m=\eta}^{\theta} f(m) \left( \lambda (I(b/k,m)) + \lambda(I'(b/k,m)) \right) , \]
where
\[ \begin{split} \lambda (I(b/k,m)) &= \left| [0;b_1,b_2,\ldots, b_s,m] - [0;b_1,b_2,\ldots, b_s,m+1] \right| \\ &= \left| \frac{m p_s + p_{s-1}}{m q_s +q_{s-1}} - \frac{(m+1) p_s + p_{s-1}}{(m+1) q_s +q_{s-1}} \right| \\ &= \frac{1}{(m q_s + q_{s-1})((m+1)q_s+q_{s-1})} , \end{split} \]
and
\[ \begin{split} \lambda (I'(b/k,m)) &= \left| [0;b_1,b_2,\ldots, b_{s-1}, b_s-1,1,m] - [0;b_1,b_2,\ldots, b_{s-1}, b_s-1,1,m+1] \right| \\ &= \left| \frac{m p_s + p_s-p_{s-1}}{m q_s + q_s - q_{s-1}} - \frac{(m+1) p_s + p_s-p_{s-1}}{(m+1) q_s + q_s - q_{s-1}} \right| \\ &= \frac{1}{((m+1)q_s-q_{s-1})((m+2)q_s-q_{s-1})} . \end{split} \]
Since $p_s=b$ and $q_s=k$, we have $b q_{s-1} \equiv (-1)^{s+1} \pmod{k}$, thus the map $b \mapsto q_{s-1}$ is a bijection of $\mathbb{Z}_k^*$. Therefore
\[ \begin{split} & \sum_{b \in \mathbb{Z}_k^*} \int_0^1 w_{f,\eta,\theta} \left( \frac{b}{k}, x \right) \, \mathrm{d} x \\ & = \sum_{m=\eta}^{\theta} \frac{f(m)}{k^2} \sum_{b \in \mathbb{Z}_k^*} \left( \frac{1}{(m+\frac{b}{k})(m+1+\frac{b}{k})} + \frac{1}{(m+1-\frac{b}{k})(m+2-\frac{b}{k})} \right) .  \end{split} \]
Applying Lemma \ref{lemma_Manuel} with $T=(\log k)^{101}$ and $u=2+\log k/(202 \log \log k)$ shows that the discrepancy of the set $\{ b/k \, : \, b \in \mathbb{Z}_k^* \}$ is $\ll 1/(\log k)^{100}$. Lemma \ref{lemma_koksma} thus gives
\[ \begin{split} \sum_{b \in \mathbb{Z}_k^*} &\left( \frac{1}{(m+\frac{b}{k})(m+1+\frac{b}{k})} + \frac{1}{(m+1-\frac{b}{k})(m+2-\frac{b}{k})} \right) \\ &= \varphi(k) \int_0^1 \left( \frac{1}{(m+x)(m+1+x)} + \frac{1}{(m+1-x)(m+2-x)} \right) \, \mathrm{d}x + O \left( \frac{k}{(\log k)^{100}m^2} \right) \\ &=2 \varphi (k) \log \left( 1+\frac{1}{m(m+2)} \right) + O \left( \frac{k}{(\log k)^{100}m^2} \right) , \end{split} \]
and the claim follows.
\end{proof}

\begin{proof}[Proof of Theorem \ref{expectedvaluetheorem}] Let $W_{k,f} = \frac{1}{\varphi (N)} \sum_{a \in \mathbb{Z}_N^*} \sum_{b \in \mathbb{Z}_k^*} w_{f,\eta,\theta} \left( \frac{b}{k} , \frac{a}{N} \right)$, $1 \le k \le N-1$. By the identity \eqref{weightfunctionidentity}, we have
\[ \frac{1}{\varphi (N)} \sum_{a \in \mathbb{Z}_N^*} S_{f,\eta,\theta} \left( \frac{a}{N} \right) = \sum_{k=1}^{N-1} W_{k,f} . \]
It follows from fundamental results of Diophantine approximation that all four endpoints of the intervals $I(b/k,m)$ and $I'(b/k,m)$ have mod 1 distance from $b/k$ between $1/((m+3)k^2)$ and $1/(mk^2)$. Therefore for all $x \in I(b/k,m) \cup I'(b/k,m)$, we have $1/((m+3)k^2) \le \| b/k-x \| \le 1/(mk^2)$. Let
\begin{equation} \label{g_k_def_rev} 
g_k(y) = \left\{ \begin{array}{ll} f(\theta) & \textrm{if } \| y \| \le \frac{1}{(\theta +1)k^2}, \\ f(m) & \textrm{if } \frac{1}{(m+1)k^2} <\| y \| \le \frac{1}{mk^2} \textrm{ with some integer } \eta \le m \le \theta, \\ 0 & \textrm{if } \frac{1}{\eta k^2} < \| y \| , \end{array} \right. 
\end{equation}
where $\| \cdot \|$ denotes the distance from the nearest integer.
This is an even, $1$-periodic function, non-increasing on $[0,1/2]$. By the assumption that $f$ is non-decreasing, it satisfies $w_{f,\eta,\theta} (b/k,x) \le g_k(b/k-x)$ for all $x \in [0,1]$, thus
\begin{equation} \label{as_here_rev}  
W_{k,f} \le \frac{1}{\varphi (N)} \sum_{a \in \mathbb{Z}_N^*} \sum_{b \in \mathbb{Z}_k^*} g_k \left( \frac{b}{k} - \frac{a}{N} \right) . 
\end{equation}
For any $a \in \mathbb{Z}_N^*$ and $b \in \mathbb{Z}_k^*$, we have $\frac{b}{k} - \frac{a}{N} = \frac{c}{\mathrm{lcm} (N,k)}$ with some integer $c \not\equiv 0 \pmod{\mathrm{lcm} (N,k)}$. The map $(a,b) \mapsto c$ is a surjective homomorphism from the additive group $\mathbb{Z}_N \times \mathbb{Z}_k$ to $\mathbb{Z}_{\mathrm{lcm} (N,k)}$, therefore each mod $\mathrm{lcm} (N,k)$ residue class is attained by exactly $(Nk)/\mathrm{lcm} (N,k)$ pairs $(a,b) \in \mathbb{Z}_N \times \mathbb{Z}_k$. Hence
\begin{equation}\label{arguing_as_here}
\begin{split} W_{k,f} &\le \frac{Nk}{\varphi (N) \mathrm{lcm} (N,k)} \sum_{c=1}^{\mathrm{lcm} (N,k)-1} g_k \left( \frac{c}{\mathrm{lcm} (N,k)} \right) \\
&\le \frac{2 Nk}{\varphi (N)} \int_0^{1/2} g_k (y) \, \mathrm{d}y \\ &= \frac{2N}{\varphi(N) k} \left( \frac{f(\theta)}{\theta+1} + \sum_{m=\eta}^{\theta} \frac{f(m)}{m(m+1)} \right) \\ &\ll \frac{B_{f,\eta,\theta} \log \log N}{k}. \end{split}
\end{equation}
The claim \eqref{expectedupperbound} immediately follows from summing over $1 \le k \le N-1$.

Assume now, that $\sqrt{N}>\theta+2$, and let us prove \eqref{expectedasymptotics}. Let $T \ge \log N$ a parameter to be chosen, and $c(T)=4 (\log T)^2 / \log \log T$. Lemma \ref{lemma_weightfunctions} gives
\begin{equation}\label{expectedupperlowerbounds} \begin{split} 
& \frac{1}{\varphi (N)} \sum_{a \in \mathbb{Z}_N^*} S_{f,\eta,\theta} \left( \frac{a}{N} \right) \\ 
& = 2 \sum_{1 \le k \le \frac{\sqrt{N}}{e^{c(T)}}} W_{k,f} + O \left( A_{f,\eta,\theta} \log \log N + \sum_{\frac{\sqrt{N}}{\max \{ \theta+2, e^{c(T)} \}}<k<\sqrt{N}} W_{k,f} \right). \end{split}
\end{equation}

We fix $b \in \mathbb{Z}_k^*$, and apply Koksma's inequality as stated in Lemma \ref{lemma_koksma} to the function $w_{f,\eta,\theta} \left( b/k , \cdot \right)$ and the point set $\{a/N$, $a \in \mathbb{Z}_N^*\}$. Note that the intervals $I(b/k,m)$, $\eta \le m \le \theta$ are contiguous, and $I'(b/k,m)$, $\eta \le m \le \theta$ are also contiguous. The function $w_{f,\eta,\theta} \left( b/k , \cdot \right)$ is thus supported on the union of two short intervals of length $\le 1/k^2$, and by the assumption that $f$ is non-decreasing, we have $V(w_{f,\eta,\theta}(b/k,\cdot);[0,1]) \le 4 f(\theta)$. The discrepancy estimate in Lemma \ref{lemma_Manuel} with $u=4 \log T / (\log \log T)$ (whence $u^{-u/2} \ll (\log N) /T$ and $T^u = e^{c(T)}$) shows that for any interval $J \subseteq [0,1]$ of length $\lambda (J) \le 1/k^2$,
\[ \sup_{I \subseteq J} \left| \frac{1}{\varphi (N)} \sum_{a \in \mathbb{Z}_N^*} \mathds{1}_{I} \left( \frac{a}{N} \right) - \lambda(I) \right| \ll \frac{\log N}{k^2 T} + \frac{e^{c(T)}}{\varphi (N)} , \]
where the supremum is taken over all subintervals of $J$. We thus obtain
\begin{equation}\label{fixkbsum}
\frac{1}{\varphi (N)} \sum_{a \in \mathbb{Z}_N^*} w_{f,\eta,\theta} \left( \frac{b}{k} , \frac{a}{N} \right) = \int_0^1 w_{f,\eta,\theta} \left( \frac{b}{k},x \right) \, \mathrm{d}x + O \left( \frac{f(\theta) \log N}{k^2 T} + \frac{f(\theta) e^{c(T)}}{\varphi (N)} \right) .
\end{equation}
By Lemma \ref{integrallemma}, summing over $b \in \mathbb{Z}_k^*$ leads to
\[ W_{k,f} = \frac{\pi^2 A_{f,\eta, \theta}}{6} \cdot \frac{\varphi (k)}{k^2} + O \left( \frac{A_{f,\eta,\theta}}{k(\log (k+1))^{100}} + \frac{f(\theta) \log N}{k T} + \frac{k f(\theta) e^{c(T)}}{\varphi (N)} \right) . \]
Standard results on partial sums of multiplicative functions (see e.g.\ \cite[Theorem 14.3]{kouk}) show that $\sum_{1 \le k \le x} \varphi (k)/k^2 = \frac{6}{\pi^2} \log x +O(1)$. The main term in \eqref{expectedupperlowerbounds} is thus
\[ 2 \sum_{1 \le k \le \frac{\sqrt{N}}{e^{c(T)}}} W_{k,f} = A_{f,\eta,\theta} \log N + O \left( A_{f,\eta,\theta} c(T) + \frac{f(\theta) (\log N)^2}{T} \right) . \]

It remains to estimate $W_{k,f}$ for large $k$. Recall that by \eqref{arguing_as_here} we have
\begin{equation*}
W_{k,f} \ll \frac{B_{f,\eta,\theta} \log \log N}{k},
\end{equation*}
which implies that
\[ \sum_{\frac{\sqrt{N}}{\max \{ \theta+2, e^{c(T) \}}} < k < \sqrt{N}} W_{k,f} \ll B_{f,\eta,\theta} (\log \theta +c(T)) \log \log N . \]

Combining the previous estimates with \eqref{expectedupperlowerbounds}, we deduce
\[ \frac{1}{\varphi (N)} \sum_{a \in \mathbb{Z}_N^*} S_{f,\eta,\theta} \left( \frac{a}{N} \right) = A_{f,\eta,\theta} \log N + O \left( B_{f,\eta,\theta} (\log \theta + c(T)) \log \log N + \frac{f(\theta) (\log N)^2}{T} \right) . \]
The claim \eqref{expectedasymptotics} follows by choosing $T = C_{f,\eta,\theta}(\log N)^2 \ge 3 (\log N)^2$.
\end{proof}

\section{Variance} \label{sec_variance}

In this section we deduce an upper bound for the variance of $S_{f,\eta,\theta}$ from Lemmas \ref{lemma_weightfunctions} and \ref{lemma_weight_square}.
\begin{thm}\label{variancetheorem} Let $N \ge 3$ and $1 \le \eta \le \theta$ be integers, and let $f: \mathbb{N} \to [0,\infty )$ be a non-decreasing function. Assume that $\theta \le \min\{\sqrt{N}-2,(\log N)^C\}$ with some constant $C>0$. Let $A_{f,\eta,\theta}$, $B_{f,\eta,\theta}$, $C_{f,\eta,\theta}$ be as in Theorem \ref{expectedvaluetheorem}, and let $D_{f,\eta,\theta}=B_{f,\eta,\theta} \sum_{m=\eta}^{\theta} f(m)/m^4$. Then
\[ \begin{split} \frac{1}{\varphi (N)} \sum_{a \in \mathbb{Z}_N^*} S_{f,\eta,\theta} \left( \frac{a}{N} \right)^2 - \bigg( \frac{1}{\varphi (N)} &\sum_{a \in \mathbb{Z}_N^*} S_{f,\eta,\theta} \left( \frac{a}{N} \right) \bigg)^2 \\ \ll &D_{f,\eta,\theta} (\log N)^2 + B_{f^2,\eta,\theta} \log N +  B_{f,\eta,\theta}^2 \log N (\log \log N)^3 \\ &+ B_{f,\eta,\theta} \left( \sqrt{B_{f^2,\eta,\theta}}+ f(\lceil \sqrt{C_{f,\eta,\theta}} \rceil) \right) \log N \log \log N \end{split} \]
with an implied constant depending only on $C$.
\end{thm}

\begin{proof} Throughout the proof we write
$$
X = \frac{1}{\varphi (N)} \sum_{a \in \mathbb{Z}_N^*} S_{f,\eta,\theta} \left( \frac{a}{N} \right)^2 \qquad \textrm{and} \qquad 
Y = \bigg( \frac{1}{\varphi (N)} \sum_{a \in \mathbb{Z}_N^*} S_{f,\eta,\theta} \left( \frac{a}{N} \right) \bigg)^2,
$$
as well as $V = X - Y$. The aim is thus to establish an upper bound for $V$. Theorem \ref{expectedvaluetheorem} shows that
\begin{equation} \label{Y}
Y = A_{f,\eta,\theta}^2 (\log N)^2 + O \left( B_{f,\eta,\theta}^2 \log N (\log \log N)^3 \right) .
\end{equation}

Now we deduce an upper bound for $X$ from Lemma \ref{lemma_weight_square}. Set $R=\lceil \sqrt{C_{f,\eta,\theta}} \rceil$, $T = C_{f,\eta,\theta}(\log N)^3$ and $c(T)=4 (\log T)^2 / (\log \log T)$, and note that in particular $C_{f,\eta,\theta} \le 5 \theta \ll (\log N)^C$, as well as $\log R \ll \log \log N$ and $c(T) \ll \frac{(\log \log N)^2}{\log \log \log N}$.

Fix $k$ in the range $1 \le k \le \sqrt{N} /e^{c(T)}$, and fix $b \in \mathbb{Z}_k^*$. As observed in formula \eqref{fixkbsum} in the proof of Theorem \ref{expectedvaluetheorem},
\[ \frac{1}{\varphi (N)} \sum_{a \in \mathbb{Z}_N^*} w_{f,\eta,\theta} \left( \frac{b}{k} , \frac{a}{N} \right) = \int_0^1 w_{f,\eta,\theta} \left( \frac{b}{k}, x \right) \, \mathrm{d}x + O \left( \frac{B_{f,\eta,\theta}}{k^2 (\log N)^2} + \frac{f(\theta) e^{c(T)}}{\varphi (N)} \right) . \]
We computed $\int_0^1 w_{f,\eta,\theta}(b/k,x) \, \mathrm{d}x$ in the proof of Lemma \ref{integrallemma}. Using the fact that the function
\[ h(y)= \frac{1}{(m+y)(m+1+y)} + \frac{1}{(m+1-y)(m+2-y)} \]
satisfies $V(h;[0,1]) \ll 1/m^4$, we get $h(y)=\int_0^1 h + O(1/m^4) = 2 \log \left( 1+\frac{1}{m(m+2)} \right) + O(1/m^4)$ uniformly in $y \in [0,1]$. Hence
\[ \int_0^1 w_{f,\eta,\theta}(b/k,x) \, \mathrm{d}x = \frac{\pi^2 A_{f,\eta,\theta}}{6k^2} + O \left( \frac{1}{k^2} \sum_{m=\eta}^{\theta} \frac{f(m)}{m^4} \right) , \]
and consequently
\[ \frac{1}{\varphi (N)} \sum_{a \in \mathbb{Z}_N^*} w_{f,\eta,\theta} \left( \frac{b}{k} , \frac{a}{N} \right) =  \frac{\pi^2 A_{f,\eta,\theta}}{6k^2} + O \left( \frac{1}{k^2} \sum_{m=\eta}^{\theta} \frac{f(m)}{m^4} + \frac{B_{f,\eta,\theta}}{k^2 (\log N)^2} + \frac{f(\theta) e^{c(T)}}{\varphi (N)} \right) . \]
We multiply the previous formula by $S_{f,\eta,\theta} (b/k)$, and then sum over $b \in \mathbb{Z}_k^*$ and $k$. By an application of Theorem \ref{expectedvaluetheorem} with the fixed denominator $k$, we obtain
\begin{equation}\label{smallkcontribution}
\begin{split} \frac{1}{\varphi (N)} &\sum_{1 \le k \le \frac{\sqrt{N}}{e^{c(T)}}} \sum_{a \in \mathbb{Z}_N^*} \sum_{b \in \mathbb{Z}_k^*} w_{f,\eta,\theta} \left( \frac{b}{k} , \frac{a}{N} \right) S_{f,\eta,\theta} \left( \frac{b}{k} \right) \\ &= \frac{\pi^2 A_{f,\eta,\theta}^2}{6} \sum_{(\theta+2)^2+1 \leq k \le \frac{\sqrt{N}}{e^{c(T)}}} \frac{\varphi(k) \log k}{k^2} + O \left( D_{f,\eta,\theta} (\log N)^2 +B_{f,\eta,\theta}^2 \log N (\log \log N)^3 \right) \\ &\leq \frac{A_{f,\eta,\theta}^2}{8} (\log N)^2 + O \left( D_{f,\eta,\theta} (\log N)^2 + B_{f,\eta,\theta}^2 \log N (\log \log N)^3 \right) . \end{split}
\end{equation}
In the last step we used $\sum_{1 \le k \le x} \varphi (k) \log k /k^2 = \frac{3}{\pi^2} (\log x)^2 + O(\log x)$, which follows from $\sum_{1 \le k \le x} \varphi (k) /k^2 = \frac{6}{\pi^2} \log x + O(1)$ via summation by parts.

Next, fix $k$ in the range $\sqrt{N}/e^{c(T)}<k < \sqrt{N}/R$, and fix $b \in \mathbb{Z}_k^*$. After extending the sum over $a$ from $\mathbb{Z}_N^*$ to $\mathbb{Z}_N$ and using a trivial discrepancy estimate for the equidistant set $\{ a/N \, : \,1 \le a \le N\}$, Koksma's inequality in Lemma \ref{lemma_koksma} yields
\begin{equation*}  \begin{split} \frac{1}{\varphi(N)} \sum_{a \in \mathbb{Z}_N^*} w_{f,\eta,\theta} \left( \frac{b}{k} , \frac{a}{N} \right) &\ll \frac{\log \log N}{N} \sum_{a=1}^N w_{f,\eta,\theta} \left( \frac{b}{k} , \frac{a}{N} \right) \\ &\ll \log \log N \left( \int_0^1 w_{f,\eta,\theta} \left( \frac{b}{k}, x \right) \, \mathrm{d}x + \frac{f(\theta)}{N} \right) \\ &\ll \log \log N \left( \frac{A_{f,\eta,\theta}}{k^2} + \frac{f(\theta)}{N} \right) . \end{split} \end{equation*}
Therefore
\[ \frac{1}{\varphi (N)} \sum_{a \in \mathbb{Z}_N^*} \sum_{b \in \mathbb{Z}_k^*} w_{f,\eta,\theta} \left( \frac{b}{k} , \frac{a}{N} \right) S_{f,\eta,\theta} \left( \frac{b}{k} \right) \ll B_{f,\eta,\theta} \varphi(k) \log k \log \log N \left( \frac{A_{f,\eta,\theta}}{k^2} + \frac{f(\theta)}{N} \right) , \]
and by summing over $\sqrt{N}/e^{c(T)} < k < \sqrt{N} /R$,
\[ \begin{split} \frac{1}{\varphi (N)} \sum_{\frac{\sqrt{N}}{e^{c(T)}} < k < \frac{\sqrt{N}}{R}} \sum_{a \in \mathbb{Z}_N^*} &\sum_{b \in \mathbb{Z}_k^*} w_{f,\eta,\theta} \left( \frac{b}{k} , \frac{a}{N} \right) S_{f,\eta,\theta} \left( \frac{b}{k} \right) \\ &\ll B_{f,\eta,\theta}^2 \log N \log \log N c(T) + \frac{B_{f,\eta,\theta}^2 C_{f,\eta,\theta}\log N \log \log N}{R^2}  \\
& \ll B_{f,\eta,\theta}^2 \log N (\log \log N)^3, \end{split} \]
since we chose $R =\lceil \sqrt{C_{f,\eta,\theta}} \rceil$. 

Next, we apply Theorem \ref{expectedvaluetheorem} to the function $f^2$. Using $C_{f^2, \eta, \theta} \le 5 \theta \ll (\log N)^C$, we obtain
\[ \frac{1}{\varphi(N)} \sum_{a \in \mathbb{Z}_N^*} S_{f^2,\eta,\theta}\left( \frac{a}{N} \right) \ll B_{f^2,\eta,\theta} \log N . \]
Arguing as in \eqref{arguing_as_here}, we deduce
\[ \begin{split} \frac{1}{\varphi(N)} \sum_{\sqrt{N}/R  \leq k <\sqrt{N}} \sum_{a \in \mathbb{Z}_N^*} \sum_{b \in \mathbb{Z}_k^*} w_{f^2,\eta,\theta} \left( \frac{b}{k} , \frac{a}{N} \right) \ll & \sum_{\sqrt{N}/R \leq k < \sqrt{N}} B_{f^2,\eta,\theta} \frac{\log \log N}{k} \\ \ll & B_{f^2,\eta,\theta} (\log \log N)^2.\end{split} \]
Observing that $\Xi_{f,\eta, \theta}$, as defined in \eqref{Xidef}, satisfies $\Xi_{f,\eta, \theta} \ll B_{f^2,\eta, \theta}$, we also have
\[ \frac{\Xi_{f,\eta, \theta} N}{\varphi(N)} \ll B_{f^2,\eta,\theta} \log \log N. \]
Finally, by Theorem \ref{expectedvaluetheorem}, we have
\[ \frac{1}{\varphi (N)} \sum_{a \in \mathbb{Z}_N^*} S_{f,\eta,\theta} (a/N) \ll B_{f,\eta,\theta} \log N . \]

The previous five formulas together with \eqref{smallkcontribution} give an upper bound for all terms which appear in Lemma \ref{lemma_weight_square}, and lead to
\[ \begin{split} X \le A_{f,\eta,\theta}^2 (\log N)^2 + O \Bigg( &D_{f,\eta,\theta} (\log N)^2+B_{f,\eta,\theta}^2 \log N (\log \log N)^3 + B_{f^2,\eta,\theta} \log N \\ &+ \sqrt{X} \sqrt{B_{f^2,\eta,\theta}} \log \log N + B_{f,\eta,\theta} f(\lceil \sqrt{C_{f,\eta,\theta}} \rceil) \log N \log \log N \Bigg). \end{split} \]
We have $V = X - Y$, and so $\sqrt{X} \leq \sqrt{V} + \sqrt{Y}$, which by \eqref{Y} implies $\sqrt{X} \ll \sqrt{V} + B_{f,\eta,\theta} \log N$. Inserting this into the previous formula, we obtain
\[ \begin{split} X \le A_{f,\eta,\theta}^2 &(\log N)^2 \\+ O \Bigg( &D_{f,\eta,\theta} (\log N)^2+B_{f,\eta,\theta}^2 \log N (\log \log N)^3 + B_{f^2,\eta,\theta} \log N \\ &+ \sqrt{V} \sqrt{B_{f^2,\eta,\theta}} \log \log N+ B_{f,\eta,\theta} \left( \sqrt{B_{f^2,\eta,\theta}} + f(\lceil \sqrt{C_{f,\eta,\theta}} \rceil) \right) \log N \log \log N \Bigg). \end{split} \]
Subtracting $Y$ from the previous formula thus leads to
\[ \begin{split}
V \ll &D_{f,\eta,\theta} (\log N)^2 +  B_{f,\eta,\theta}^2 \log N (\log \log N)^3 + B_{f^2,\eta,\theta} \log N \\
&+ \sqrt{V} \sqrt{B_{f^2,\eta,\theta}} \log \log N + B_{f,\eta,\theta} \left( \sqrt{B_{f^2,\eta,\theta}} + f(\lceil \sqrt{C_{f,\eta,\theta}} \rceil) \right) \log N \log \log N . \end{split}\]
This is an inequality of the form $V \ll c + \sqrt{V} d$, which implies $V \ll c + d^2$. Thus
\[\begin{split} V \ll &D_{f,\eta,\theta} (\log N)^2 +  B_{f,\eta,\theta}^2 \log N (\log \log N)^3 + B_{f^2,\eta,\theta} \log N \\ &+ B_{f,\eta,\theta} \left( \sqrt{B_{f^2,\eta,\theta}}+ f(\lceil \sqrt{C_{f,\eta,\theta}} \rceil) \right) \log N \log \log N, \end{split} \]
as claimed.
\end{proof}

\section{Proof of the main theorems}\label{sec_mainproofs}

\begin{proof}[Proof of Theorems \ref{Mtheorem} and \ref{Ltheorem}] Let $f(x)=1$. Note that with the notation of Theorem \ref{expectedvaluetheorem}, we have $A_{f,\eta,\theta} = \frac{12}{\pi^2} \sum_{m=\eta}^{\theta} \log \left( 1+\frac{1}{m(m+2)} \right)$ and $B_{f,\eta,\theta} \ll 1/\eta$. If $c \le (\log N)^{C+1}$, then by choosing $\eta=b$ and $\theta=c$ we have $L_{[b,c]}(a/N)=S_{f,\eta,\theta}(a/N)$, and Theorem \ref{expectedvaluetheorem} immediately yields \eqref{Ltexpectedvalue}. If $c>(\log N)^{C+1}$, then by choosing $\eta=b$ and $\theta=\lfloor (\log N)^{C+1} \rfloor$ we have $S_{f,\eta,\theta}(a/N) \le L_{[b,c]}(a/N) \le S_{f,\eta,\theta}(a/N)+L_{[\theta+1,N]}(a/N)$, and Theorem \ref{expectedvaluetheorem} yields
\[ \frac{1}{\varphi(N)} \sum_{a \in \mathbb{Z}_N^*} L_{[b,c]} \left( \frac{a}{N} \right) = \mu_{[b,\theta]} \log N + O \left( \frac{(\log \log N)^3}{b} + \frac{\log N \log \log N}{\theta} \right) . \]
Here $\mu_{[b,\theta]} = \mu_{[b,c]}+O(1/\theta)$, and $\frac{\log N \log \log N}{\theta} \ll (\log \log N)^3/b$. This finishes the proof of \eqref{Ltexpectedvalue}.

Observing that $M(a/N) \ge t \log N \, \Leftrightarrow \, L_{[\lceil t \log N \rceil,N]}(a/N) \ge 1$, an application of the Markov inequality leads to Theorem \ref{Mtheorem}: for any $0<t \le (\log N)^C$, we have
\[ \begin{split} \frac{1}{\varphi (N)} \bigg| \bigg\{ a \in \mathbb{Z}_N^* \, : \, &M \left( \frac{a}{N} \right) \ge t \log N \bigg\} \bigg| \\ &\le \frac{1}{\varphi (N)} \sum_{a \in \mathbb{Z}_N^*} L_{[\lceil t \log N \rceil,N]} \left( \frac{a}{N} \right) \\ &= \frac{12}{\pi^2} \log N \sum_{m \ge t \log N} \log \left( 1+\frac{1}{m(m+2)} \right) + O \left( \frac{(\log \log N)^3}{t \log N} \right) \\ &=\frac{12}{\pi^2 t} + O \left( \frac{(\log \log N)^3}{t \log N} \right) . \end{split} \]

It remains to prove the upper bound for the variance of $L_{[b,c]}(a/N)$. With the notation of Theorem \ref{variancetheorem} we have $D_{f,\eta,\theta} \ll 1/\eta^4$. If $c< (\log N)^{4C+1}$, then choosing $\eta=b$ and $\theta=c$ we have $L_{[b,c]}(a/N)=S_{f,\eta,\theta}(a/N)$, and Theorem \ref{variancetheorem} immediately yields
\[ \frac{1}{\varphi(N)} \sum_{a \in \mathbb{Z}_N^*} \left( L_{[b,c]} \left( \frac{a}{N} \right) -E \right)^2 \ll \frac{(\log N)^2}{b^4} + \frac{\log N \log \log N}{b} \]
with $E=\varphi(N)^{-1} \sum_{a \in \mathbb{Z}_N^*} L_{[b,c]}(a/N)=\mu_{[b,c]} \log N +O((\log \log N)^3/b)$. The previous formula thus also holds with $E$ replaced by $\mu_{[b,c]} \log N$, as claimed.

If $c> (\log N)^{4C+1}$, then we choose $\eta=b$ and $\theta = \lceil b^4 \log N \rceil$. An application of Theorem \ref{variancetheorem} similarly leads to
\[ \frac{1}{\varphi (N)} \sum_{a \in \mathbb{Z}_N^*} \left( S_{f, \eta,\theta} \left( \frac{a}{N} \right) - E \right)^2 \ll \frac{(\log N)^2}{b^4} + \frac{\log N \log \log N}{b} \]
with $E=\varphi (N)^{-1} \sum_{a \in \mathbb{Z}_N^*} S_{f,\eta,\theta} (a/N) = \mu_{[b,c]} \log N + O((\log \log N)^3/b)$. The previous formula thus also holds with $E$ replaced by $\mu_{[b,c]} \log N$. On the other hand, we have $L_{[b,c]} (a/N) \neq S_{f, \eta, \theta} (a/N) \, \Leftrightarrow \, M(a/N) > \theta$, hence Theorem \ref{Mtheorem} gives
\[ \frac{1}{\varphi (N)} \left| \left\{ a \in \mathbb{Z}_N^* \, : \, L_{[b,c]} \left( \frac{a}{N} \right) \neq S_{f, \eta, \theta} \left( \frac{a}{N} \right) \right\} \right| \ll \frac{1}{b^4} . \]
Using the uniform bound $L_{[b,c]} (a/N) \ll \log N$, the previous two formulas establish the desired upper estimate for the variance.
\end{proof}

\begin{proof}[Proof of Theorem \ref{Stheorem}] The theorem trivially holds for $0<t<100$, hence we may assume that $100 \le t \le (\log N)^C$. Set $U = \lfloor \log_2 (t \log N) \rfloor$, and let $1 \le u \le U$ be an integer. Let $f(x)=x$, $\eta =2^{u-1}$ and $\theta=2^u -1$, and consider $S^{(u)} (a/N) := S_{f,\eta, \theta} (a/N)$. Then
$$
A_{f,\eta,\theta} = \frac{12 \log 2}{\pi^2} + O(2^{-u}), \quad B_{f,\eta,\theta} \ll 1, \quad  C_{f,\eta,\theta} \ll 2^u, \quad D_{f,\eta,\theta} \ll 2^{-2u},
$$
as well as $B_{f^2,\eta,\theta} \ll 2^u$, and $C_{f^2,\eta,\theta} \ll 2^u$. By Theorems \ref{expectedvaluetheorem} and \ref{variancetheorem}, for all $1 \le u \le U$, we have
\[ \frac{1}{\varphi(N)} \sum_{a \in \mathbb{Z}_N^*}  S^{(u)} \left( \frac{a}{N} \right) = \frac{12 \log 2}{\pi^2} \log N  + O \left( 2^{-u} \log N + (\log \log N)^3 \right) , \]
and
\begin{equation}\label{dyadic_variance} \frac{1}{\varphi (N)} \sum_{a \in \mathbb{Z}_N^*}  \left( S^{(u)} \left( \frac{a}{N} \right) - \frac{12 \log 2}{\pi^2} \log N \right)^2 \ll 2^u \log N + 2^{-2u} (\log N)^2 . \end{equation}
An application of the Chebyshev inequality thus gives
\[ \begin{split} \frac{1}{\varphi (N)} \bigg| \bigg\{ a \in \mathbb{Z}_N^* \, : \, \bigg| S^{(u)} \left( \frac{a}{N} \right)  - \frac{12 \log 2}{\pi^2} &\log N \bigg| \geq \frac{t \log N}{10 (\min\{u,U-u\})^{2}}  \bigg\} \bigg| \\ & \ll \left( 2^u \log N + 2^{-2u} (\log N)^2 \right) \cdot \frac{(\min\{u,U-u\})^{4}}{t^2 (\log N)^2} \\ & \ll \frac{2^u (U-u)^4}{t^2 \log N} + \frac{2^{-2u} u^4}{t^2} . \end{split} \]
Since
$$
\sum_{u=1}^U \frac{1}{10(\min\{u,U-u\})^{2}} \leq \pi^2/30 \leq 1/2,
$$ 
the union bound immediately yields
\[ \begin{split} \frac{1}{\varphi (N)} \bigg| \bigg\{ a \in \mathbb{Z}_N^* \, : \, \bigg| \sum_{u=1}^U S^{(u)} \left( \frac{a}{N} \right)  - \frac{12 \log 2}{\pi^2} U &\log N \bigg| \geq \frac{t \log N}{2} \bigg\} \bigg| \\ &\ll \sum_{u=1}^U \left( \frac{2^u (U-u)^4}{t^2 \log N} + \frac{u^{4} 2^{-2u}}{t^2} \right) \ll  \frac{1}{t} .
\end{split}\]

On the other hand,
\[ S \left( \frac{a}{N}\right) = \sum_{u=1}^U S^{(u)} \left( \frac{a}{N}\right) + \sum_{i=1}^r \mathds{1}_{\{2^{U} \leq a_i\}} a_i , \]
and in particular, $S(a/N) \neq \sum_{u=1}^U S^{(u)} (a/N) \, \Leftrightarrow \, M(a/N) \geq 2^U$. Theorem \ref{Mtheorem} thus shows that
\[ \frac{1}{\varphi (N)} \left| \left\{ a \in \mathbb{Z}_N^* \, : \, S \left( \frac{a}{N} \right) \neq \sum_{u=1}^U S^{(u)} \left( \frac{a}{N}\right) \right\} \right| \ll \frac{1}{t} , \]
and consequently,
\[ \frac{1}{\varphi (N)} \left| \left\{ a \in \mathbb{Z}_N^* \, : \, \left| S \left( \frac{a}{N} \right)  - \frac{12 \log 2}{\pi^2} U \log N \right| \geq \frac{t \log N}{2} \right\} \right| \ll  \frac{1}{t} . \]
Finally, note that
$$
\frac{12}{\pi^2} |\log \log N - U \log 2| \leq \frac{12}{\pi^2} (1 + \log t) \leq t/2
$$
holds by the assumption $t \ge 100$. Overall, we conclude that
\[ \frac{1}{\varphi (N)} \left| \left\{ a \in \mathbb{Z}_N^* \, : \, \left| S \left( \frac{a}{N} \right)  - \frac{12}{\pi^2} \log N \log \log N \right| \geq t \log N \right\} \right| \ll \frac{1}{t},
\]
as desired.
\end{proof}

\section{Dedekind sums}\label{sec_dedekind}

By the Barkan--Hickerson formula \eqref{hickerson}, Theorem \ref{Dtheorem} is equivalent to showing that
\begin{equation}\label{need_to_show}
\frac{1}{\varphi(N)} \left| \left\{a \in \mathbb{Z}_N^*: \left| S_{\mathrm{alt}} \left( \frac{a}{N} \right) \right| \geq t \log N\right\} \right| \ll \frac{1}{t}
\end{equation}
for all $0 < t \le (\log N)^C$. Note that $S_{\mathrm{alt}}(a/N) = \sum_{i=1}^r (-1)^i a_i$ cannot be written in the form
$\sum_{i=1}^r f(a_i)$ due to the presence of the alternating factor $(-1)^i$, so the framework developed in the previous sections does not directly apply to this problem. We will first outline the necessary changes in comparison to the arguments developed in the previous sections, before actually proving in Lemma \ref{var_ded} a second moment estimate which allows us to deduce \eqref{need_to_show}.

Fixing $1 \leq \eta \leq \theta$, we write
\begin{equation}\label{def_Se}
	S_{\textup{e}}(a/N) =  \sum_{i=1}^{\lfloor r/2 \rfloor} \mathds{1}_{\{\eta \leq a_{2i} \leq \theta\}}a_{2i}, \quad S_{\textup{o}}(a/N) =  \sum_{i=0}^{\lceil r/2 \rceil -1} \mathds{1}_{\{\eta \leq a_{2i+1} \leq \theta\}}a_{2i+1}
\end{equation}
for partial sums corresponding to even resp.\ odd indices, and
$$
S_{\mathrm{alt},\eta,\theta}(a/N) = S_{\textup{e}}(a/N) - S_{\textup{o}}(a/N) = \sum_{i = 1}^r (-1)^i \mathds{1}_{\{\eta \leq a_{i} \leq \theta\}}a_i.
$$
We would like to reduce $S_{\textup{e}}(a/N)$ and $S_{\textup{o}}(a/N)$ to counting only partial quotients associated with convergent denominators of size up to roughly $\sqrt{N}$, by taking into account the reflection via $a^*$ as in the proof of Lemma \ref{lemma_weightfunctions}.
This leads to
\begin{equation}\label{heur_even}S_{\textup{e}}(a/N) \approx
	\sum_{\substack{i =1 \\ q_{2i-1}(a/N) < \sqrt{N}}}^{\lfloor r/2 \rfloor }\mathds{1}_{\{\eta \leq a_{2i} \leq \theta\}}a_{2i}
	+ \sum_{\substack{i =1 \\ q_{r-2i}(a^*/N) < \sqrt{N}}}^{\lfloor r/2 \rfloor } \mathds{1}_{\{\eta \leq a_{2i} \leq \theta\}}a_{2i}
	\end{equation}
and a similar expression for $S_{\textup{o}}(a/N)$. However, when trying to rewrite the second sum on the right-hand side of \eqref{heur_even} in terms of the partial quotients of $a^*/N$, similar to equation \eqref{symmetry}, then we encounter a problem which arises from the parity condition: re-indexing gives us either
	\begin{equation}\label{a_in_A}
	\sum_{\substack{i =1 \\ q_{r-2i}(a^*/N) < \sqrt{N}}}^{\lfloor r/2 \rfloor} \mathds{1}_{\{\eta \leq a_{2i} \leq \theta\}}a_{2i}
	\approx 
	\sum_{\substack{i =1 \\ q_{2i-1}(a^*/N) < \sqrt{N}}}^{\lfloor r/2 \rfloor}  \mathds{1}_{\{\eta \leq a^*_{2i} \leq \theta\}}a^*_{2i}
	\end{equation}	
	or	
		\begin{equation}\label{a_in_B}
	\sum_{\substack{i =1 \\ q_{r-2i}(a^*/N) < \sqrt{N}}}^{\lfloor r/2 \rfloor}  \mathds{1}_{\{\eta \leq a_{2i} \leq \theta\}}a_{2i}
	\approx 
	\sum_{\substack{i =1 \\ q_{2i-2}(a^*/N) < \sqrt{N}}}^{\lfloor r/2 \rfloor}  \mathds{1}_{\{\eta \leq a^*_{2i-1} \leq \theta\}}a^*_{2i-1},
	\end{equation}
	depending on the parity of the length $r$, and on whether $a< N/2$ or not. The same happens for
	$S_{\textup{o}}$. So writing
	
\begin{equation}
	\label{def_e'}
		S'_{\textup{e}}(a/N) = \sum_{\substack{i =1 \\ 2 \leq q_{2i-1} < \sqrt{N}}}^{\lfloor r/2 \rfloor} \mathds{1}_{\{\eta \leq a_{2i} \leq \theta\}}a_{2i}, \quad	S'_{\textup{o}}(a/N) = \sum_{\substack{i =1 \\ 2 \leq q_{2i} < \sqrt{N}}}^{\lceil r/2 \rceil -1}\mathds{1}_{\{\eta \leq a_{2i+1} \leq \theta\}} a_{2i+1},
\end{equation}
	we have
	
	\[S_{\mathrm{alt},\eta,\theta}(a/N) \approx S'_{\textup{e}}(a/N) - S'_{\textup{o}}(a/N) + S'_{\textup{e}}(a^*/N) - S'_{\textup{o}}(a^*/N)\]
	if $a$ falls into the case where \eqref{a_in_A} holds, whereas 
	
	\[S_{\mathrm{alt},\eta,\theta}(a/N) \approx S'_{\textup{e}}(a/N) - S'_{\textup{o}}(a/N) - S'_{\textup{e}}(a^*/N) + S'_{\textup{o}}(a^*/N)\]
if we are in case $\eqref{a_in_B}$.
We will show that the average values of $S'_{\textup{e}}, S'_{\textup{o}}$ are equal and cancel out, so the key point will be to bound the variances
of $S'_{\textup{e}}, S'_{\textup{o}}$. 
We will see that it actually suffices to consider $S'_{\textup{e}}$.
Note that the parity of $i$ is equivalent to whether $\frac{p_i}{q_i}$ approximates $\frac{a}{N}$ 
from the left or from the right. Hence, we can count partial quotients with even index by 
taking one-sided weight functions $w_{\textup{e}}\left(b/k,\cdot\right)$ which are only supported on reals $x$ that are
smaller than $b/k$ (in contrast to the two-sided weight functions in Section \ref{sec_weight}). This will allow us to obtain a non-trivial estimate
on the variance of $S'_{\textup{e}}$, and lead to the desired bound on the second moment of $S_{\mathrm{alt}}$.

We note that this method would allow us to obtain estimates on the mean and the variance of more general sums
\[\sum_{\substack{1 \leq i \leq r,\\ i \textup{ even}}} f(a_{i})\mathds{1}_{\{\eta \leq a_{i} \leq \theta\}} \quad \text{and} \quad \sum_{\substack{1 \leq i \leq r,\\ i \textup{ odd}}} f(a_{i})\mathds{1}_{\{\eta \leq a_{i} \leq \theta\}},\]
with $f$ being a non-decreasing function. However, we can only distinguish between even and odd indices of partial quotients, so that
statistics for sums of a more general form such as
$\sum\limits_{1 \leq i \leq r/3} a_{3i}$
would need a different approach.

\begin{lemma}\label{var_ded}
	Let $1 \leq \eta \leq \theta$ be integers such that $\theta \le (\log N)^C$ for some constant $C > 0$. Then we have
	\begin{equation*}\label{var_ded_eq}
	\begin{split}
	\frac{1}{\varphi(N)} \sum_{a \in \mathbb{Z}_N^*} S_{\mathrm{alt},\eta,\theta}\left(\frac{a}{N}\right)^2 
	\ll \frac{1 + \log \frac{\theta}{\eta}}{\eta}(\log N)^2 + \theta \log N + \left(1 + \log \frac{\theta}{\eta}\right)^2 \log N (\log \log N)^3,
	\end{split}
	\end{equation*}
	where the implied constant only depends on $C$.
\end{lemma}

\begin{proof}
Let $S_{\textup{e}}, S_{\textup{o}}, S'_{\textup{e}}, S'_{\textup{o}}$ be defined as in \eqref{def_Se} and \eqref{def_e'}. For the rest of the proof we assume that $N$ is so large that $\theta \leq (\log N)^C < \sqrt{N}-2$, as otherwise the conclusion of the lemma holds trivially. Given $a \in \mathbb{Z}_N^*$, let $a^*$ be defined as in the proof of Lemma \ref{lemma_weightfunctions}. Assume that $a < N/2$. Recall that in this case $a^*_i = a_{r-i+1}$ for $1 \leq i \leq r$.
If $r$ is odd, then similar to \eqref{symmetry} re-indexing yields
\begin{equation} \label{caseodd}
S_{\textup{e}}'(a^*/N) =
\sum_{\substack{i=1 \\ q_{r-2i}(a^*/N) < \sqrt{N}}}^{\lfloor r/2\rfloor} \mathds{1}_{\{\eta \leq a_{2i} \leq \theta \}} a_{2i}
\end{equation} 
and
\begin{equation} \label{caseodd_B}
S_{\textup{o}}'(a^*/N) =
\sum_{\substack{i=1 \\ q_{r-2i-1}(a^*/N) < \sqrt{N}}}^{\lfloor r/2\rfloor-1} \mathds{1}_{\{\eta \leq a_{2i+1} \leq \theta \}} a_{2i+1},
\end{equation}
where we used the fact that $q_0(a^*/N) < 2$ and for $q_{r-1}(a^*) < \sqrt{N}, a_r^{*} \geq \sqrt{N}-1 > \theta$.
If $r$ is even, by a similar reasoning
\begin{equation}\label{caseeven_A}
	S_{\textup{e}}'(a^*/N) = \sum_{\substack{i=1 \\ q_{r-2i+1}(a^*/N) < \sqrt{N}}}^{r/2 -1} \mathds{1}_{\{\eta \leq a_{2i+1} \leq \theta \}} a_{2i+1}
\end{equation} 
and
\begin{equation}\label{caseeven_B}
	S_{\textup{o}}'(a^*/N) = \sum_{\substack{i=1 \\ q_{r-2i}(a^*/N) < \sqrt{N}}}^{r/2-1} \mathds{1}_{\{\eta \leq a_{2i} \leq \theta \}} a_{2i}.
\end{equation}

Turning to the case $a > N/2$, observe that the conditions $2 \leq q_i(a/N) < \sqrt{N}$, together with $a_1 = 1$ and $\theta < \sqrt{N}-2$, imply 
that $a_1,a_2,a_r$ do not contribute to $S_{\textup{e}}'(a/N), S_{\textup{o}}'(a/N)$.
Since $a_{i}^* = a_{r-i+2}$ for $3 \leq i \leq r-1$,
it follows for odd $r$ that
\begin{equation}\label{caseodd2_A}
		S_{\textup{e}}'(a^*/N) = \sum_{\substack{i=1 \\ q_{r-2i}(a^*/N) < \sqrt{N}}}^{\lfloor r/2\rfloor -1} \mathds{1}_{\{\eta \leq a_{2i+1} \leq \theta \}} a_{2i+1}
\end{equation}
and
\begin{equation} \label{caseodd2_B}
		S_{\textup{o}}'(a^*/N) = \sum_{\substack{i=2 \\ q_{r-2i+1}(a^*/N) < \sqrt{N}}}^{\lfloor r/2\rfloor} \mathds{1}_{\{\eta \leq a_{2i} \leq \theta \}} a_{2i},
\end{equation}
whereas for even $r$ we have
\begin{equation}\label{caseeven2_A}
		S_{\textup{e}}'(a^*/N) =
		\sum_{\substack{i=1 \\ q_{r-2i-1}(a^*/N) < \sqrt{N}}}^{r/2-1} \mathds{1}_{\{\eta \leq a_{2i} \leq \theta \}} a_{2i}
\end{equation}
and
\begin{equation} \label{caseeven2_B}
		S_{\textup{o}}'(a^*/N) =
		\sum_{\substack{i=1 \\ q_{r-2i}(a^*/N) < \sqrt{N}}}^{r/2 -1} \mathds{1}_{\{\eta \leq a_{2i+1} \leq \theta \}} a_{2i+1}.
\end{equation}

We partition $\mathbb{Z}_N^*$ into the sets $A,B$ which are defined by
\begin{equation*}
	\begin{split}A &:= \left\{a \in \mathbb{Z}_N^*: a < \frac{N}{2} \text{ with } r(a/N) \text{ odd, or }
		a > \frac{N}{2} \text{ with } r(a/N) \text{ even}\right\},\\
		B &:= \left\{a \in \mathbb{Z}_N^*: a < \frac{N}{2} \text{ with } r(a/N) \text{ even, or }
		a > \frac{N}{2} \text{ with } r(a/N) \text{ odd}\right\}.
	\end{split}
\end{equation*}

We see that if $a \in A$, we have \eqref{caseodd} and \eqref{caseodd_B}, or \eqref{caseeven2_A} and \eqref{caseeven2_B}, so that

\begin{equation}\label{caseA_A}
	\begin{split}
		& \left\lvert
		S_{\textup{e}}(a/N)
		 - S_{\textup{e}}'(a/N) - S_{\textup{e}}'(a^*/N) \right\rvert \\
		& \qquad \leq \mathds{1}_{\{ a_1=1 \} } \mathds{1}_{\{\eta \leq a_{2} \leq \theta \}}a_2 + \mathds{1}_{\{\eta \leq a_{r} \leq \theta \}}a_r + \sum_{\substack{i = 1\\ \frac{\sqrt{N}}{\theta +2} < q_{i-1} < \sqrt{N}}}^{r}\mathds{1}_{\{\eta \leq a_{i} \leq \theta \}} a_{i},
    \end{split}
\end{equation}
and
\begin{equation}\label{caseA_B}
	\begin{split}
		& \left\lvert S_{\textup{o}}(a/N)
		 - S_{\textup{o}}'(a/N) - S_{\textup{o}}'(a^*/N) \right\rvert \\
		& \qquad \leq \mathds{1}_{\{\eta \leq a_{1} \leq \theta \}}a_1 + \mathds{1}_{\{\eta \leq a_{r} \leq \theta \}}a_r + \sum_{\substack{i = 1\\ \frac{\sqrt{N}}{\theta +2} < q_{i-1} < \sqrt{N}}}^{r}\mathds{1}_{\{\eta \leq a_{i} \leq \theta \}} a_{i}.
	\end{split}
\end{equation}

Here we used the fact (proven in Lemma \ref{lemma_weightfunctions}) that
for $a < N/2$ and $a_i \leq \theta$, $q_{i-1}(a/N) < \sqrt{N}, q_{r-i}(a^*/N) < \sqrt{N}$ implies $\frac{\sqrt{N}}{\theta +2} < q_{i-1}(a/N) < \sqrt{N}$,
and the same holds for $a > N/2$ with $q_{r-i}(a^*/N)$ replaced by $q_{r-i+1}(a^*/N)$.
If $a \in B$, we have \eqref{caseeven_A} and \eqref{caseeven_B}, or \eqref{caseodd2_A} and \eqref{caseodd2_B}, so that by similar arguments we get

\begin{equation}\label{caseB_A}
	\begin{split}
		& \left\lvert
		S_{\textup{e}}(a/N)
		 - S_{\textup{e}}'(a/N) - S_{\textup{o}}'(a^*/N) \right\rvert \\
		& \qquad \leq \mathds{1}_{\{a_1=1\}} \mathds{1}_{\{\eta \leq a_{2} \leq \theta \}}a_2 + \mathds{1}_{\{\eta \leq a_{r} \leq \theta \}}a_r + \sum_{\substack{i = 1 \\ \frac{\sqrt{N}}{\theta +2} < q_{i-1} < \sqrt{N}}}^{r}\mathds{1}_{\{\eta \leq a_{i} \leq \theta \}} a_{i}
\end{split}
\end{equation}
and
\begin{equation} \begin{split} \label{caseB_B}
		& \left\lvert S_{\textup{o}}(a/N)
		- S_{\textup{o}}'(a/N) - S_{\textup{e}}'(a^*/N) \right\rvert \\
		&  \qquad \leq \mathds{1}_{\{\eta \leq a_{1} \leq \theta \}}a_1 + \mathds{1}_{\{\eta \leq a_{r} \leq \theta \}}a_r + \sum_{\substack{i = 1\\ \frac{\sqrt{N}}{\theta +2} < q_{i-1} < \sqrt{N}}}^{r}\mathds{1}_{\{\eta \leq a_{i} \leq \theta \}} a_{i}.
	\end{split}
\end{equation}

Motivated by \eqref{caseA_A}--\eqref{caseB_B}, we define 

\begin{equation}\label{def_sum'}
	S_{\mathrm{alt}}'\left(\frac{a}{N}\right) :=\begin{cases}
		S'_{\textup{e}}(a/N) - S'_{\textup{o}}(a/N) + S'_{\textup{e}}(a^*/N) - S'_{\textup{o}}(a^*/N) &\text{ if } a \in A,\\
		S'_{\textup{e}}(a/N) - S'_{\textup{o}}(a/N) + S'_{\textup{o}}(a^*/N) - S'_{\textup{e}}(a^*/N) &\text{ if } a \in B.
	\end{cases}
\end{equation}
Clearly, we have

\begin{equation}\label{split_error}
	\frac{1}{\varphi(N)} \sum_{a \in \mathbb{Z}_N^*} S_{\mathrm{alt}}\left(\frac{a}{N}\right)^2 
	\ll 
	\frac{1}{\varphi(N)} \sum_{a \in \mathbb{Z}_N^*} S_{\mathrm{alt}}'\left(\frac{a}{N}\right)^2
	+	\frac{1}{\varphi(N)} \sum_{a \in \mathbb{Z}_N^*} R\left(\frac{a}{N}\right)^2,
\end{equation}
where $R(a/N) := |S_{\mathrm{alt}}'(a/N) - S_{\mathrm{alt}}(a/N)|$.
Starting with the last sum in the previous formula, we can deduce from \eqref{caseA_A}--\eqref{caseB_B} that for all $a \in \mathbb{Z}_N^*$ we have

\begin{equation}\label{beginR} R\left(\frac{a}{N}\right) \ll \sum_{\substack{i = 1\\ \frac{\sqrt{N}}{\theta +2} < q_{i-1} < \sqrt{N}}}^{r}\mathds{1}_{\{\eta \leq a_{i} \leq \theta \}} a_{i}
	+  \mathds{1}_{\{\eta \leq a_1 \leq \theta\}}a_1
	+ \mathds{1}_{\{a_1 = 1\}}\mathds{1}_{\{\eta \leq a_2 \leq \theta\}}a_2
	+  \mathds{1}_{\{\eta \leq a_r \leq \theta\}}a_r.
\end{equation}
	Arguing as in Lemma \ref{lemma_weightfunctions}, we have with $f(x) = x$ that
	
	\begin{equation} \begin{split} \label{err_estim_2} \mathds{1}_{\{\eta \leq a_1 \leq \theta\}}a_1
	+ \mathds{1}_{\{a_1 = 1\}}\mathds{1}_{\{\eta \leq a_2 \leq \theta\}}a_2
	+  \mathds{1}_{\{\eta \leq a_r \leq \theta\}}a_r
	\\ \qquad \leq w_{f,\eta,\theta}\left(1,\frac{a}{N}\right) + w_{f,\eta,\theta}\left(1,\frac{a^*}{N}\right), \end{split}
	\end{equation}
	where $w_{f,\eta,\theta}$ is defined as in \eqref{def_weight}. With an argument similar to the one leading to \eqref{smallai} we obtain
	
	\begin{equation}\label{err_estim_1}\begin{split}
	&  \sum_{\substack{i=1\\ \frac{\sqrt{N}}{\theta +2}< q_{i-1}(a/N)<\sqrt{N}}}^r \mathds{1}_{\{\eta \leq a_i \leq \theta\}}a_i
	 \\ &\ll \sum_{\substack{i=1\\ \frac{\sqrt{N}}{\theta +2}< q_{i-1}(a/N)<\sqrt{N}}}^r \mathds{1}_{\{1 \leq a_i \leq \sqrt{\theta+2}\}} a_i
	+ \sum_{\substack{i=1\\ \frac{\sqrt{N}}{\theta +2}< q_{i-1}(a/N)<\sqrt{N}}}^r \mathds{1}_{\{\sqrt{\theta+2} \leq a_i \leq \theta\}}a_i\\
	&\ll  \sqrt{\theta} \sum_{\substack{i =1 \\ \frac{\sqrt{N}}{\theta +2} \leq q_{i-1}<\sqrt{N}}}^r 1
	+ \sum_{\substack{i=1\\ \frac{\sqrt{N}}{\theta +2}< q_{i-1}(a/N)<\sqrt{N}}}^r \mathds{1}_{\{\sqrt{\theta+2} \leq a_i \leq \theta\}}a_i
	\\&\ll  \sqrt{\theta}\log \theta + \sum_{\substack{i=1\\ \frac{\sqrt{N}}{\theta +2}< q_{i-1}(a/N)<\sqrt{N}}}^r \mathds{1}_{\{\sqrt{\theta+2} \leq a_i \leq \theta\}}a_i.
	\end{split}
	\end{equation}
	Using \eqref{beginR}--\eqref{err_estim_1}, an application of the Cauchy--Schwarz inequality, and the observation that
	the last sum in \eqref{err_estim_1}
	 contains at most $2$ terms, we get
	
	\begin{align*}
	& \frac{1}{\varphi(N)}\sum_{a \in \mathbb{Z}_N^*}R\left(\frac{a}{N}\right)^2 \\ &\ll \theta (\log \theta)^2+ \frac{1}{\varphi(N)}\sum_{a \in \mathbb{Z}_N}\left(\sum_{\frac{\sqrt{N}}{\theta+2} < k < \sqrt{N}}\sum_{b \in \mathbb{Z}_k^*}w_{f^2,\eta,\theta}\left(\frac{b}{k},\frac{a}{N}\right) + w_{f^2,\eta,\theta}\left(1,\frac{a}{N}\right) \right)
	\\&\ll \theta (\log \log N)^2+ W_{1,f^2} + \sum_{\frac{\sqrt{N}}{\theta+2} < k < \sqrt{N}} W_{k,f^2},
	\end{align*}
    where $W_{k,f^2}$ is as in the proof of Theorem \ref{expectedvaluetheorem}.
    We apply \eqref{arguing_as_here} to obtain
	
	\begin{equation*}\label{var_R}
	\frac{1}{\varphi(N)}\sum_{a \in \mathbb{Z}_N^*}R\left(\frac{a}{N}\right)^2 \ll \theta (\log \log N)^2.
	\end{equation*}

		
Now we come to the first term on the right-hand side of \eqref{split_error}. Note that for $a/N < 1/2$ we have $1-a/N = [0;1,a_1(a/N)-1,\ldots,a_r(a/N)]$, which implies
\[q_{1}(1-a/N) = 1, \quad q_{i}(1-a/N) = q_{i-1}(a/N) \quad \text{for } i \geq 2.\]
Thus for every function $g$ and every set $S \subseteq \{2, \ldots, N\}$ we have
\begin{equation} \label{g_1}
\sum_{\substack{i=1 \\ q_{2i-1}(a/N) \in S}}^{\lfloor r(a/N)/2\rfloor} g(a_{2i}(a/N)) = \sum_{\substack{i=1 \\ q_{2i}(1-a/N) \in S}}^{\lceil r(1-a/N)/2\rceil-1} g(a_{2i+1}(1-a/N))
\end{equation} 
and
\begin{equation} \label{g_2}
			\sum_{\substack{i=1 \\ q_{2i+1}(1-a/N) \in S}}^{\lfloor r(1-a/N)/2\rfloor -1} g(a_{2i+2}(1-a/N))
			=
			 \sum_{\substack{i=1 \\ q_{2i}(a/N) \in S}}^{\lceil r(a/N)/2\rceil-1} g(a_{2i+1}(a/N)).
	\end{equation}
Here we used the fact that the restrictions $q_{2i-1} \geq 2$ resp.\ $q_{2i} \geq 2$ remove the contribution of $a_1$, and if $a_1 = 1$, also the contribution of $a_2$. 
In particular, \eqref{g_1} and \eqref{g_2} imply that $\sum_{a \in \mathbb{Z}_N^*}S_{\textup{e}}'(a/N)=  \sum_{a \in \mathbb{Z}_N^*}S_{\textup{o}}'(a/N)$, so we can deduce that for $a \in A$
	\begin{equation*}
	\begin{split}\label{cs_trick}
	S_{\mathrm{alt}}'(a/N)^2 &= 
	\left(S'_{\textup{e}}(a/N) + S'_{\textup{e}}(a^*/N) -  S'_{\textup{o}}(a/N) - S'_{\textup{o}}(a^*/N)\right)^2
	\\&= \Bigg(S'_{\textup{e}}(a/N)-\frac{1}{\varphi(N)}\sum_{b \in \mathbb{Z}_N^*}S_{\textup{e}}'(b/N)  + S'_{\textup{e}}(a^*/N)-\frac{1}{\varphi(N)} \sum_{b \in \mathbb{Z}_N^*}S_{\textup{e}}'(b/N) \\&- S'_{\textup{o}}(a/N)+\frac{1}{\varphi(N)}\sum_{b \in \mathbb{Z}_N^*}S_{\textup{o}}'(b/N) - S'_{\textup{o}}(a^*/N)+\frac{1}{\varphi(N)}\sum_{b \in \mathbb{Z}_N^*}S_{\textup{o}}'(b/N)\Bigg)^2
	\\&\leq 4\left(S'_{\textup{e}}(a/N)-\frac{1}{\varphi(N)}\sum_{b \in \mathbb{Z}_N^*}S_{\textup{e}}'(b/N)\right)^2  + 4\left(S'_{\textup{e}}(a^*/N)-\frac{1}{\varphi(N)}\sum_{b \in \mathbb{Z}_N^*}S_{\textup{e}}'(b/N)\right)^2 
	\\&+ 4\left(S'_{\textup{o}}(a/N)-\frac{1}{\varphi(N)}\sum_{b \in \mathbb{Z}_N^*}S_{\textup{o}}'(b/N)\right)^2 + 4\left(S'_{\textup{o}}(a^*/N)-\frac{1}{\varphi(N)}\sum_{b \in \mathbb{Z}_N^*}S_{\textup{o}}'(b/N)\right)^2,
	\end{split}
	\end{equation*}
	where we used the Cauchy--Schwarz inequality in the last step. By an analogous argument, the same formula also holds for $a \in B$.
	Since $a \mapsto a^*$ is a bijection on $\mathbb{Z}_N^*$, this implies
		\begin{equation}\label{var_reduced}
	\begin{split}
	\frac{1}{\varphi(N)}\sum_{a \in \mathbb{Z_N^*}} S_{\mathrm{alt}}'(a/N)^2
	\leq\; & 8 	\frac{1}{\varphi(N)}\sum_{a \in \mathbb{Z_N^*}} \left(S'_{\textup{e}}(a/N)-\frac{1}{\varphi(N)}\sum_{b \in \mathbb{Z}_N^*}S_{\textup{e}}'(b/N)\right)^2
	\\+\; & 8 	\frac{1}{\varphi(N)}\sum_{a \in \mathbb{Z_N^*}} \left(S'_{\textup{o}}(a/N)-\frac{1}{\varphi(N)}\sum_{b \in \mathbb{Z}_N^*}S_{\textup{o}}'(b/N)\right)^2
	\\=\; &16\left(\frac{1}{\varphi(N)}\sum_{a \in \mathbb{Z_N^*}} S'_{\textup{e}}(a/N)^2-\left(\frac{1}{\varphi(N)}\sum_{b \in \mathbb{Z}_N^*}S_{\textup{e}}'(b/N)\right)^2\right),
	\end{split}
	\end{equation}
where we used \eqref{g_1} and \eqref{g_2} in the last step.
It remains to bound the variance of $S'_e$. For the rest of this section, we fix $f(x) = x$, although the same procedure works for general non-decreasing $f$. 
Note that
\[
\sum_{a \in \mathbb{Z}_N^*} S_{\textup{e}}'(a/N) = \sum_{a \in \mathbb{Z}_N^*} S_{\textup{o}}'(a/N)
= \frac{1}{2}\sum_{2 \leq k < \sqrt{N}} w_{f,\eta,\theta}\left(\frac{b}{k},\frac{a}{N}\right).
\]
Using \eqref{upper_weightfct} and \eqref{k=1_estimate} as well as Theorem \ref{expectedvaluetheorem}, we obtain

\begin{equation}\label{ded_expsq}\left(\frac{1}{\varphi(N)}\sum_{a \in \mathbb{Z}_N^*}S_{\textup{e}}'(a/N)\right)^2
= \frac{A_{f,\eta,\theta}^2}{16} (\log N)^2 + O\left(B_{f,\eta,\theta}^2
\log N (\log \log N)^3\right).
\end{equation}

Defining $S_{\textup{e}}^{(\ge 2)}(a/N) := \sum\limits_{\substack{i=1\\ q_{2i-1} \ge 2}}^{\lfloor r/2 \rfloor} \mathds{1}_{\{\eta \leq a_{2i} \leq \theta\}}a_{2i}$
and applying \eqref{almost_Sbk}, we obtain

\[\begin{split}&\sum_{a \in \mathbb{Z}_N^*} \sum_{\substack{i = 1\\
			2 \leq q_{2i-1} < \sqrt{N}}}^{\lfloor r/2 \rfloor} \mathds{1}_{\{\eta \leq a_{2i} \leq \theta\}}a_{2i}
	\sum_{\substack{j = 1\\
			2 \leq q_{2j-1} < \sqrt{N}}}^{ \lfloor r/2 \rfloor}  \mathds{1}_{\{\eta \leq a_{2j} \leq \theta\}}a_{2j}
	\\&\quad\leq 2\sum_{a \in \mathbb{Z}_N^*} \sum_{\substack{i = 1\\
			2 \leq q_{2i-1} < \frac{\sqrt{N}}{e^{c(T)}}}}^{\lfloor r/2 \rfloor}\mathds{1}_{\{\eta \leq a_{2i} \leq \theta\}}a_{2i}\cdot
	S_{\textup{e}}^{(\ge 2)}\left(\frac{p_{2i-1}}{q_{2i-1}}	\right)
\\ &\quad+ \sum_{a \in \mathbb{Z}_N^*}\sum_{i=1}^r  \mathds{1}_{\{\eta \leq a_{i} \leq \theta\}}a_i^2 + 2\sum_{a \in \mathbb{Z}_N^*} \sum_{\substack{i = 1\\
		\frac{\sqrt{N}}{e^{c(T)}} \leq q_{i-1} < \sqrt{N}}}^r\mathds{1}_{\{\eta \leq a_{i} \leq \theta\}}a_{i}
\left( S_{f,\eta,\theta}\left(\frac{p_{i-1}}{q_{i-1}}\right)
+ \mathds{1}_{\{a_{i-2} = \theta\}}\theta
\right)
 \\ &\quad+2\sum_{a \in \mathbb{Z}_N^*} \sum_{\substack{i = 1\\ 2 \leq q_{i-1} < \frac{\sqrt{N}}{e^{c(T)}}}}^{\lfloor r/2 \rfloor}\mathds{1}_{\{\eta \leq a_{i} \leq \theta\}}a_{i}\mathds{1}_{\{a_{i-2} = \theta\}}\theta,
\end{split}\]
where $T$ and $c(T)$ are chosen as in the proof of Theorem \ref{variancetheorem}.
Everything except the first term in this equation can be estimated as in the proof of Lemma \ref{lemma_weight_square} resp.\ Theorem \ref{variancetheorem}.
Counting partial quotients with even indices is equivalent to only considering convergents $b/k$ for which $b/k - a/N > 0$, because $p_{2i}/q_{2i} < a/N < p_{2i-1}/q_{2i-1}$. 
	Accordingly we define $I_{\textup{left}}(b/k,m)$ to be the one interval out of $I(b/k,m)$, $I'(b/k,m)$, that lies to the left of $b/k$. That is,
	
	\[I_{\textup{left}}(b/k,m) := \begin{cases}
		I(b/k,m) &\text{ if } 	I(b/k,m) \subseteq [0,b/k],\\
		I'(b/k,m) &\text{ if } I(b/k,m) \subseteq [b/k,1].
	\end{cases}\]
	Further, we define the one-sided weight function
	
	\[w_{\textup{e},\eta,\theta} \left( \frac{b}{k}, x \right) = \sum_{m=\eta}^{\theta} m   \cdot \mathds{1}_{I_{\textup{left}}(b/k,m)} \left(x \right)\]
	to obtain
	
	\[
	\sum_{a \in \mathbb{Z}_N^*} \sum_{\substack{i = 1\\
						2 \leq q_{2i-1} < \sqrt{N}}}^{\lfloor r/2\rfloor }\mathds{1}_{\{\eta \leq a_{2i} \leq \theta\}}a_{2i}\cdot
				S_{\textup{e}}^{(\ge 2)}\left(\frac{p_{2i-1}}{q_{2i-1}}\right) 
				= \sum_{a \in \mathbb{Z}_N^*} \sum_{2 \leq k < \sqrt{N}}
					\sum_{b \in \mathbb{Z}_k^*}w_{\textup{e},\eta,\theta} \left( \frac{b}{k}, \frac{a}{N} \right)
					S_{\textup{e}}^{(\ge 2)} \left(\frac{b}{k}\right).
	\]
	Using an analogue of \eqref{fixkbsum} and
	\[
	\int_{0}^1 w_{\textup{e},\eta,\theta} \left( \frac{b}{k}, x \right)  \, \mathrm{d} x = \frac{\pi^2 A_{f,\eta,\theta}}{12k^2} + O\left(
	\frac{1}{k^2}\sum_{m = \eta}^{\theta} \frac{f(m)}{m^3}
	\right),
	\]
	we get
	\[
	\frac{1}{\varphi(N)}\sum_{a \in \mathbb{Z}_N^*} w_{\textup{e},\eta,\theta} \left( \frac{b}{k}, \frac{a}{N} \right)
	= \frac{\pi^2 A_{f,\eta,\theta}}{12k^2} + O\left(
	\frac{1}{k^2}\sum_{m = \eta}^{\theta} \frac{f(m)}{m^3}+ \frac{f(\theta) \log N}{k^2 T} + \frac{f(\theta) e^{c(T)}}{\varphi (N)} \right).
	\]
	Summing over $b$ and $k$ and applying Theorem \ref{expectedvaluetheorem} for fixed denominator $k$, together with \eqref{g_1} and \eqref{g_2} this shows
	analogously to \eqref{smallkcontribution}
	that 
	\[\begin{split}
	&\sum_{a \in \mathbb{Z}_N^*} \sum_{2 \leq k < \frac{\sqrt{N}}{e^{c(T)}}}
	\sum_{b \in \mathbb{Z}_k^*}w_{\textup{e},\eta,\theta} \left( \frac{b}{k}, \frac{a}{N} \right)
	S_{\textup{e}}^{(\ge 2)} \left(\frac{b}{k}\right)
	\\&\quad \leq \frac{A_{f,\eta,\theta}^2}{32} (\log N)^2 + O \left( D'_{f,\eta,\theta} (\log N)^2 + B_{f,\eta,\theta}^2 \log N (\log \log N)^3 \right),
	\end{split}
	\]
	where $D'_{f,\eta,\theta} := B_{f,\eta,\theta}\sum_{m = \eta}^{\theta} \frac{f(m)}{m^3}.$
	Combining the previous estimates proves the desired result.
	\end{proof}

\begin{proof}[Proof of Theorem \ref{Dtheorem}]
	We can now prove Theorem \ref{Dtheorem} in the precisely same fashion as 
	Theorem \ref{Stheorem}, with the only difference that
   \eqref{dyadic_variance} is replaced by
	\begin{equation}\label{dyadic_ded}\frac{1}{\varphi(N)} \sum_{a \in \mathbb{Z}_N^*} S_{\mathrm{alt},2^{u-1},2^u -1}\left(\frac{a}{N}\right)^2 
	\ll 2^u \log N + \frac{(\log N)^2}{2^u},\end{equation}
	which we obtain by an application of Lemma \ref{var_ded} with $\eta = 2^{u-1}, \theta = 2^u -1$.
\end{proof}

\section*{Acknowledgments}

CA is supported by the Austrian Science Fund (FWF), projects F-5512, I-4945, I-5554, P-34763, P-35322 and Y-901. BB is supported by FWF project M-3260. The authors want to thank Gerhard Larcher, whose comments on his paper \cite{larcher} were the starting point for their work resulting in the present paper. The authors also want to thank Sandro Bettin and Sary Drappeau for several interesting discussions on the topic of the paper. Furthermore, they want to thank Alexey Ustinov for many valuable remarks, and for his help to access some of the Russian-language literature. Finally, they want to thank two anonymous referees who read the paper carefully and provided valuable suggestions for improvements.

\bibliography{CF}
\bibliographystyle{abbrv}

\end{document}